\documentclass{article}

\usepackage{arxiv}
\usepackage{amsmath}
\usepackage{amssymb}
\usepackage[utf8]{inputenc} 
\usepackage[T1]{fontenc}    
\usepackage{hyperref}       
\usepackage{url}            
\usepackage{booktabs}       
\usepackage{amsfonts}       
\usepackage{nicefrac}       
\usepackage{microtype}      
\usepackage{lipsum}
\usepackage{xspace}
\usepackage{breqn}
\usepackage{algorithm}
\usepackage{algorithmic}
\usepackage{graphicx}
\usepackage{amsthm}
\usepackage{mathrsfs}
\usepackage{mathtools}
\usepackage{tikz}

\newtheorem{lemma}{Lemma}[section]
\newtheorem{theorem}[lemma]{Theorem}
\newtheorem{corollary}{Corollary}

\newtheorem*{remark*}{Remark}
\newtheorem{assumption}{Assumption}
\newtheorem{definition}{Definition}

\newcommand{\method}{\textsf{SPAN}\xspace}

\newcommand{\grad}{\ensuremath{\nabla}}

\makeatletter
\newcommand\mathcircled[1]{%
  \mathpalette\@mathcircled{#1}%
}
\newcommand\@mathcircled[2]{%
  \tikz[baseline=(math.base)] \node[draw,circle,inner sep=1pt] (math) {$\m@th#1#2$};%
}

\title{SPAN: A Stochastic Projected Approximate\\ Newton Method}

\author{
  Xunpeng Huang\\
  Bytedance AI Lab\\
  Beijing, China\\
  \texttt{huangxunpeng@bytedance.com}
    \And
  Xianfeng Liang\\
  Department of Computer Science\\
  University of Science and Technology China\\
  Hefei, China\\
  \texttt{zeroxf@mail.ustc.edu.cn}
    \And
  Zhengyang Liu\\
  Bytedance AI Lab\\
  Beijing, China\\
  \texttt{liuzhengyang.lozycs@bytedance.com}\\
    \And
  Yitan Li\\
  Bytedance AI Lab\\
  Beijing, China\\
  \texttt{liyitan@bytedance.com}\\
    \And
  Linyun Yu\\
  Bytedance AI Lab\\
  Beijing, China\\
  \texttt{yulinyun@bytedance.com}\\
    \And
  Yue Yu\\
  Department of Computer Science\\
  Tsinghua University\\
  Beijing, China\\
  \texttt{yu-y14@mails.tsinghua.edu.cn}
    \And
  Lei Li\thanks{Corresponding author}\\
  Bytedance AI Lab\\
  Beijing, China\\
  \texttt{lileilab@bytedance.com}
}

\begin{document}
\maketitle

\begin{abstract}
Second-order optimization methods have desirable convergence properties.
However, the exact Newton method requires expensive computation for the Hessian and its inverse.
In this paper, we propose \method, a novel approximate and fast Newton method.
\method computes the inverse of the Hessian matrix via low-rank approximation and stochastic Hessian-vector products.
Our experiments on multiple benchmark datasets demonstrate that \method outperforms existing first-order and second-order optimization methods in terms of the convergence wall-clock time.
Furthermore, we provide a theoretical analysis of the per-iteration complexity, the approximation error, and the convergence rate.
Both the theoretical analysis and experimental results show that our proposed method achieves a better trade-off between the convergence rate and the per-iteration efficiency.
\end{abstract}

\section{Introduction}
\label{sec:intro}
Mathematical optimization plays an important role in machine learning.
Many learning tasks can be formulated as a problem of minimizing a finite sum objective:
\begin{equation}
    \label{eq:obj_sum}
    \min\limits_{x\in \mathbb{R}^d} F(x) \mathop{=}^{\Delta} \frac{1}{N}\mathop{\sum}\limits_{i=1}^N f_i(x),
\end{equation}
where $N$ is the  number of samples, $d$ is the dimension of parameters and $f_i(x)$ denotes the loss function for sample $i$.
In order to solve Eq.\eqref{eq:obj_sum}, many first- and second-order methods have been proposed and has the following update paradigm:
\begin{equation}
    \label{eq:newton_update}
        x_{t+1} = x_t - \eta_t H^{-1}(x_t)g(x_t), t=0,1,2,\ldots,
\end{equation}
where $g(x_t)$ is the gradient and  $\eta_t$ is the step size at $t$-th iteration.
The term $H^{-1}(x_t)$ can be set differently in different methods.
First-order methods set $H^{-1}(x_t)$ as an identity matrix.
The resulting updating procedure becomes gradient descent (GD) or stochastic gradient descent (SGD)
depending on whether the gradient $g(x_t)$ is calculated over the whole sample set or one random sample~\cite{robbins1985stochastic,li2014efficient,cotter2011better}.
A series of first-order linearly convergent methods and their variance reduction variants were proposed to accelerate the above iteration updates, including SVRG~\cite{johnson2013accelerating}, SAGA~\cite{defazio2014saga}, SDCA~\cite{shalev2014accelerated}, etc.
Their main intuition is to balance the computational cost of $g(x_t)$ and the approximation gap between the stochastic gradient $g(x_t)$ and the global gradient $\nabla F(x_t)$.

Compared with first-order optimizers, second-order optimization methods regard $H^{-1}(x_t)$ in Eq.~\eqref{eq:newton_update} as the Hessian inverse (the standard Newton Method) or certain carefully constructed Hessian inverse (quasi-Newton methods).
The matrix $H^{-1}(x_t)$ can be thought to adjust the step size and gradient coordinates through the high order information or the quasi-Newton condition.
Second-order methods usually achieve a better convergence rate than first-order ones.

However, second-order algorithms are not widely used in large-scale optimization problems due to
expensive computation cost.
Computing the Hessian matrix and its inverse requires high computation complexity and consumes large memory.
Standard Newton method takes $\mathcal{O}(Nd^2+d^3)$ per iteration, where $N$ and $d$ are the number of samples and the number of unknown parameters, respectively.
Quasi-Newton methods like BFGS and L-BFGS~\cite{liu1989limited} are faster,
but still requires $\mathcal{O}(Nd+d^2)$.
However, they do not maintain the local quadratic convergence rate.
For problems with large parameters, it is not affordable even if the inverse calculation could fit in memory.

Our goal is to develop an approximate yet efficient Newton method with provable convergence guarantee.
We aim to reduce per-iteration computation cost for the Hessian inverse calculation while maintaining the approximation accuracy.
To this end, we propose \emph{Stochastic Projected Approximate Newton} method (\method), a novel and generic  approach to speed up second-order optimization calculation.
The main contributions of the paper are as follows:
\begin{itemize}
    \item We propose a novel second-order optimization method, \method, to achieve a better trade-off between Hessian approximation error and per-iteration efficiency. Inside the method, we propose a stochastic sampling technique to construct the Hessian approximately.
    Since it only requires first-order oracles and Hessian-vector products, the Hessian and its inverse can be computed very efficiently.
    \item We present a theoretical analysis about the Hessian approximation error and the convergence rate. \method achieves linear-quadratic convergence with a provable Hessian approximation error bound.
    \item We conduct experiments on multiple real datasets against existing state-of-the-art methods. The results validate that our proposed method achieves state-of-the-art performance in terms of the wall-clock time, with almost no sacrificing on the construction accuracy of the Hessian.
\end{itemize}

\section{Related Work}
\label{sec:related}
In order to improve the per-iteration efficiency, several stochastic second-order methods have tried to seek the trade-off between per-iteration computational complexity and convergence rate, which can be divided into two main categories.
\paragraph{Stochastic Newton Methods.} A series of sub-sampled Newton methods are proposed to solve the problems where the sample size $N$ is far larger than the feature size $d$. In these methods, the Hessian matrix is approximated with a small subset of training samples.
NewSamp~\cite{erdogdu2015convergence} and similar methods by~\cite{roosta2016sub,byrd2011use} sample from function set $\{f_i\}$ randomly and construct a regularized $m$-rank approximate Hessian with truncated singular value decomposition
to improve the per-iteration efficiency. However, they need to compute the sub-sampled Hessian, which has an $\mathcal{O}(d^2)$ computational cost even if there is only one single instance. Removing the requirements of second-order oracle in NewSamp-type methods, our \method further accelerates the iteration with stochastic low-rank approximation techniques, which improves the complexity by a factor nearly $\mathcal{O}(d/m)$.

Sketch Newton method in~\cite{pilanci2017newton} adopts sketching techniques to approximate Hessian.
A non-uniform probability distribution was introduced to sample rows of $\sqrt{\nabla^2 F(x)}$ in~\cite{xu2016sub}.
Both of them use new approximate Hessian construction and proper sampling methods to improve Hessian estimation efficiency. Differ from such sketch Newton methods, our \method does not need the \emph{Hessian Decomposition} assumption for sketched updates in~\cite{xu2016sub,pilanci2017newton}.

Rather than estimating the Hessian, LiSSA by~\cite{agarwal2017second} approximates the Hessian inverse with matrix Taylor expansion.
The most elegant part of LiSSA is that the Hessian inverse estimation is unbiased, and the approximation error only depends on the matrix concentration inequalities. Compared with LiSSA, \method guarantees more robust per-iteration complexity when the objective function is not \emph{Generalized Linear Model} (GLM).  Moreover, we do not require the initial solution is close to the optimum. Additionally, the procedure of calculating the approximate Hessian  for each sample in \method is independent, which means our method has better parallelism compared with LiSSA.

\paragraph{Stochastic Quasi-Newton Methods.} The quasi-Newton methods can also be improved with stochastic approximation techniques.
S-LBFGS by~\cite{moritz2016linearly} adopts the randomization to the classical L-BFGS and integrates the widely used gradient variance reduction techniques.
Subsequently, SB-BFGS in~\cite{gower2016stochastic} extends BFGS with matrix sketching to approximate Hessian inverse.
Stochastic quasi-Newton methods have a better per-iteration complexity compared with stochastic Newton methods. However, most of them cannot explicitly demonstrate the benefit of introducing the curvature information in the convergence analysis. Such problem is solved in our \method by establishing the Hessian approximation error bound which can hardly be analyzed in stochastic quasi-Newton methods.

\section{The Proposed \method Method}
\label{sec:method}
\begin{table}[t]
    \centering
    \small
    \begin{tabular}{ lll}
    \toprule
        Symbols & Description\\
    \midrule
        $N$ & The number of samples in~Eq.\eqref{eq:obj_sum}\\
        $d$ & The dimension of decision variables in~Eq.\eqref{eq:obj_sum}\\
        $B$ & A subset with $B \subseteq [N] = \{1,2,\ldots,N\}$ \\
        $b$ & The size of $B$ with $b = \left|B\right|$\\
        $F$ & The objective function in~Eq.\eqref{eq:obj_sum}\\
        $f_i(x)$ & The loss function for sample $i$ in~Eq.\eqref{eq:obj_sum}\\
        $f_B(x)$ & The batch loss with $f_B(x)=\frac{1}{b}\sum_{i\in B} f_i(x)$\\
        $g_B(x)$ & The gradient of batch loss with $g_B(x) =\grad f_B(x)$\\
        $H_B(x)$ & The batch Hessian with $H_B(x) = \grad^2 f_B(x_t)$\\
        $\sigma_i(A)$ & The $i$-th top non-zero singular vector of matrix $A$\\
        $p_i(A)$ & The singular vector of matrix $A$ corresponding to $\sigma_i(A)$\\
        $\mathrm{rank}(A)$ & The rank number of matrix $A$\\
        $\left\|\cdot\right\|$ & The Euclidean norm of a vector or $L_2$ norm of a matrix\\
    \bottomrule
    \end{tabular}
    \caption{\small Important mathematical notations in this paper.}
    \label{tab:notations}
\end{table}

In this section, we will first present the overall idea and intuition of our proposed \method.
Then, we will describe three technical components and the full algorithm. For better illustration, we list some important notations and their descriptions in Table~\ref{tab:notations}.

The goal of \method method is to optimize Eq.~\eqref{eq:obj_sum} with respect to the decision variable $x$ (i.e., model parameters) using the mini-batched iterative update:
\begin{equation}
\label{eq:batch_newton_update}
x_{t+1} = x_t - \eta_t H_B^{-1}(x_t)g_B(x_t),
\end{equation}
where $t$ is the iteration index.
Note that the straightforward calculation of Eq.~\eqref{eq:batch_newton_update} requires $\mathcal{O}(Nd + bd^2 + d^3)$ which is time-consuming for models with a large $d$.

To make the computation faster, instead of calculating the batch Hessian $H_B$ (the abbreviation of $H_B(x_t)$) and then taking the inverse explicitly, our idea is to replace $H_B$ in Eq.~\eqref{eq:batch_newton_update} with an approximate Hessian $\hat{H}_B$.
Ideally, $\hat{H}_B$ should be bounded within a small region around the true Hessian $H_B$.
We propose to use the projected approximation for the Hessian $\hat{H}_B = PH_BP^T$ where $P$ is a carefully constructed orthogonal projector~\cite{horn2012matrix}.
To further ensure  $\hat{H}_B$ invertible, we add an additional perturbation term $\Delta H$, $\hat{H}_B = PH_BP^T + \Delta H$.
Note that such formulation is sufficient to approximate $H_B$, since we can consider $P^* =\Sigma_{i=1}^kp_i(H_B)p_i(H_B)^T$ to obtain the optimal $k$-rank approximation of $H_B$.
However, in practice, we can hardly find $p_i(H_B)$ exactly without knowing the batch Hessian.
Thus, it is challenging to construct $P$ with desired efficiency and approximation accuracy.
At the first glance, this seems infeasible, while our main intuition is based on the observation that affine transformation $A$ over random vectors tends to have larger components on the top singular vectors of $A$. Such intuition is helpful to find an accurate approximation of orthonormal projector $P^*$. More details follow later.

In the remaining of this section, we present the construction of an approximate Hessian and the resulting optimization algorithm.
In particular, we design the algorithm components guided by the following questions.
\begin{enumerate}
	\item How to design a structure for $P$ so that $\hat{H}_B = PH_B P^T$ can be computed efficiently without knowing $H_B$?
	\item How to make the inverse robust --- permitting $\hat{H}_B$ to be invertible in any circumstance?
	\item How to balance the Hessian approximation error $\|\hat{H}_B - H_B\|$ and iteration complexity flexibly?
\end{enumerate}

\subsection{Stochastic Projected Approximation}
\label{sec:projected-hessian}
We construct the approximate Hessian as $\hat{H}_B = P H_B P^T$ which can be calculated without knowing $H_B$.
A proper $P$ essentially decides the direction where the Hessian $H_B$ would project onto.
We decompose $P$ into the product of the form $P = U U^T$, where $U\in \mathbb{R}^{d\times l}$ is an orthonormal matrix.
With that, we construct an approximate Hessian $\hat{H}_B$ by
\begin{equation}
    \label{eq:hessian_low_rank_approximation}
    \hat{H}_B(x) = U U^T H_B(x)U U^T.
\end{equation}
One can easily verify that the Moore–Penrose inverse of $\hat{H}_B$ can be computed efficiently via $\hat{H}_B^{\dagger}= U\left(U^TH_BU\right)^{-1}U^T$
since the size of $U^T H_B U\in \mathbb{R}^{l\times l}$ is much smaller than $H_B$ and can be calculated by the product of $U^T$ and $H_B U$.

In the following, we develop a method to obtain a $U$ while keeping $\|\hat{H}_B - H_B\|$ small.
We derive a method to calculate $H_B U$ without requiring Hessian $H_B$.

Our method is to randomly choose a set of column vectors $\Omega$ and project them to the space expanded by the singular vectors of $H_B$. From the projection, one can extract an orthonormal basis $U$ to expand a low-dimensional space for $H_B$ to project to.
The procedure is inspired by the \emph{Proto-Algorithm}~\cite{halko2011finding}, combining with the secant equation for Hessian-vector products calculation.
We utilize the standard Gaussian distribution to generate such $\Omega\in  \mathbb{R}^{d\times l}$ and calculate the projection $Y=H_B \Omega$ directly.
To make the computation more efficient, the set of random vectors $\Omega$ only contains $l$ elements where $l\ll d$.

Notice that for any matrix $V$ whose $i$-th column vector is presented as $v_i$, the extended Hessian-vector product on a sample batch $B$, denoted as $\Psi_B(x,V)$, is
\begin{align*}
    \Psi_B(x,V) &\triangleq H_{B}(x) V = \left[\begin{matrix}H_B(x)v_1 & H_B(x)v_2 & \ldots \end{matrix}\right],\\
    H_B(x)v_i &= \frac{d}{d\alpha} g_{B}(x+\alpha v_i)\Big|_{\alpha= 0}.
\end{align*}
In practice, the Hessian-vector product $H_B(x)v_i$ can be calculated by the finite difference of gradients like Algorithm~\ref{alg:hessian_vector_product}.

$U$ is then constructed as follows:
\begin{enumerate}
    \item calculate the extended Hessian-vector product of $\Omega$ to get $Y=H_B(x) \Omega=\Psi_B(x, \Omega)$;
    \item calculate the basis $U$ via QR decomposition $Y = U R$.
\end{enumerate}

With the above-constructed $U$, $\hat{H}_B(x)$ can be further constructed again using the extended Hessian-vector product, $\hat{H}_B(x)=U U^T \Psi_B(x, U) U^T$.
But there is no need to calculate $\hat{H}_B$ explicitly, it suffices to calculate its Moore-Penrose inverse.
One can verify that with such constructed $\hat{H}_B$, the error in the term of $\|\hat{H}_B(x)-H_B(x)\|$ is well-bounded~\cite{halko2011finding}.

\begin{algorithm}[tb]
   \caption{The Hessian-Vector Product}
   \label{alg:hessian_vector_product}
\begin{algorithmic}
   \STATE {\bfseries Input:} $F$, $B$, $\hat{x}$, $v$
    \STATE Choose a large constant $C \rightarrow \infty$
    \STATE Calculate residual perturbation $\hat{v} = v/C$
    \STATE Form the gradient difference $\delta = g_B(\hat{x}+\hat{v}) - g_B(\hat{x})$
   \STATE {\bfseries Return:} $C\delta$
\end{algorithmic}
\end{algorithm}

\subsection{Robust Hessian Inversion}
\label{sec:robust_hessian_inversion}
The above constructed approximate Hessian is simplified with one flaw ($\hat{H}_B$ is actually not invertible) since the orthonormal matrix $U$ is low-rank.
This can be alleviated with a perturbed version of Eq.\eqref{eq:hessian_low_rank_approximation}.
\begin{equation}
    \label{eq:perturb-hessian}
    \hat{H}_B(x) = U U^T H_B(x) U U^T + \lambda\left(I - U U^T\right)
\end{equation}
where $\lambda$ is a carefully chosen constant with $\lambda > 0$.

The purpose of the perturbation term $\lambda \left(I - UU^T\right)$ is to introduce the invertibility of the approximate matrix $\hat{H}_B(x)$ presented in Eq.~\eqref{eq:hessian_low_rank_approximation}.
We will give an informal analysis here to show the importance of choosing a proper $\lambda$.
On one hand, a large $\lambda$ will impair captures of the main actions of $H_B$, since it increases the lower bound of the Hessian approximation error  $\|\hat{H}_B -H_B\|$ as
\begin{equation*}
    \begin{split}
        & \left\|UU^TH_B(x)U U^T +  \lambda\left(I - U U^T\right)-H_B(x)\right\| \\
    \ge & \left\|\lambda\left(I - U U^T\right)\right\| - \left\|UU^TH_B(x)U U^T -H_B(x)\right\|.
    \end{split}
\end{equation*}

On the other hand, from an iteration perspective, we can neither choose a tiny $\lambda$ since it will lose the benefit of the curvature information induced by approximate Newton.
In particular, if we regard the SVDs of $I-UU^T$ and $U^TH_{B}(x)U$ as follows 
\begin{equation*}
    I-UU^T = U_{\perp}U_{\perp}^T,\quad  U^TH_{B}(x)U = \hat{U}\Lambda \hat{U}^T.
\end{equation*}
The SVD of constructed Hessian $\hat{H}_B(x_t)$ in Eq.~\eqref{eq:perturb-hessian} and its inverse can be formulated as
\begin{equation*}
    \begin{split}
        \hat{H}_B(x) &= \left[\begin{matrix}U\hat{U} & U_{\perp}\end{matrix}\right]\left[\begin{matrix}\Lambda & 0 \\ 0 & \lambda I\end{matrix}\right]\left[\begin{matrix}\hat{U}^TU^T \\ U_{\perp}^T\end{matrix}\right]\\
        \hat{H}_B^{-1}(x) &= \left[\begin{matrix}U\hat{U} & U_{\perp}\end{matrix}\right]\left[\begin{matrix}\Lambda^{-1} & 0 \\ 0 & \lambda^{-1} I\end{matrix}\right]\left[\begin{matrix}\hat{U}^TU^T \\ U_{\perp}^T\end{matrix}\right].
    \end{split}
\end{equation*}
It can be observed that a tiny $\lambda$ will make the singular vectors associated with $\lambda^{-1}$ dominate the action of $\hat{H}_B^{-1}(x)$, which impairs the introduction of curvature information taken by $U\hat{U}$ and $\Sigma$.

Hence, a proper $\lambda$ is needed to balance the Hessian approximation error and the $l_2$ norm of constructed Hessian inverse. Specifically, we will further discuss $\lambda$(see Theorem~\ref{thm:phantom_convergence_frame}) and the Hessian approximation error $\|\hat{H}_B(x)-H_B(x)\|$(see Lemma~\ref{thm:l2_norm_upper_bound}) in Section~\ref{sec:analysis}. Rigorous proof will be deferred to our supplementary materials.

\subsection{Power Iteration}
\label{sec:better_approxiamtion}
In this part, we introduce a power iteration technique to balance $\|\hat{H}_B - H_B\|$ and iteration complexity more flexibly. That is to say,
with limit iteration complexity sacrificing, the construction of $\hat{H}_B$ in Eq.~\eqref{eq:perturb-hessian} can be improved through some auxiliary steps in the main algorithm.
Generally speaking, any matrix $U$ obtained as in Section~\ref{sec:projected-hessian} is valid for $\hat{H}_B$'s construction.
However, in practice, a better orthogonal projector $UU^T$ maintains that top singular vectors of $H_B(x_t)$ are rotated less after performing the projection on $H_B(x_t)$, e.g., $UU^TH_B(x_t)UU^T$.
The matrix $U$ is usually obtained from the projection $H_B(x_t)\Omega$ of random vectors.
Thus, we hope basis vectors of the projection have larger components on top singular vectors e.g., $p_1(H_B), p_2(H_B)$, compared with those on $p_{d-1}(H_B)$ and $p_d(H_B)$.
Such a requirement can be satisfied by taking the power of the Hessian.

Specifically, as shown in Figure~\ref{fig:power_iteration} showed, normalized $H_B\Omega$ (green dots) seems to like a unit ball, while they degenerate to a unit circle (red dots) when random vectors $\Omega$ are projected through $H_B^4$ (red dots)  because the component of projection on $[1, 0, 0]$ almost have become $0$. A similar phenomenon also happens for the component of projection on $[0, 1, 0]$ when taking a larger power for Hessian, i.e., $H_B^{10}\Omega$ (yellow dots). As a result, the normalized projections nearly collapse to two points (yellow dots), i.e., $[0, 0, 1]$ and $[0, 0, -1]$. That is to say, the component of projection on bottom singular vectors tends to vanish. Besides, such a phenomenon becomes more significant as the power of Hessian increases.
\begin{figure}
    \includegraphics[width=0.5\textwidth]{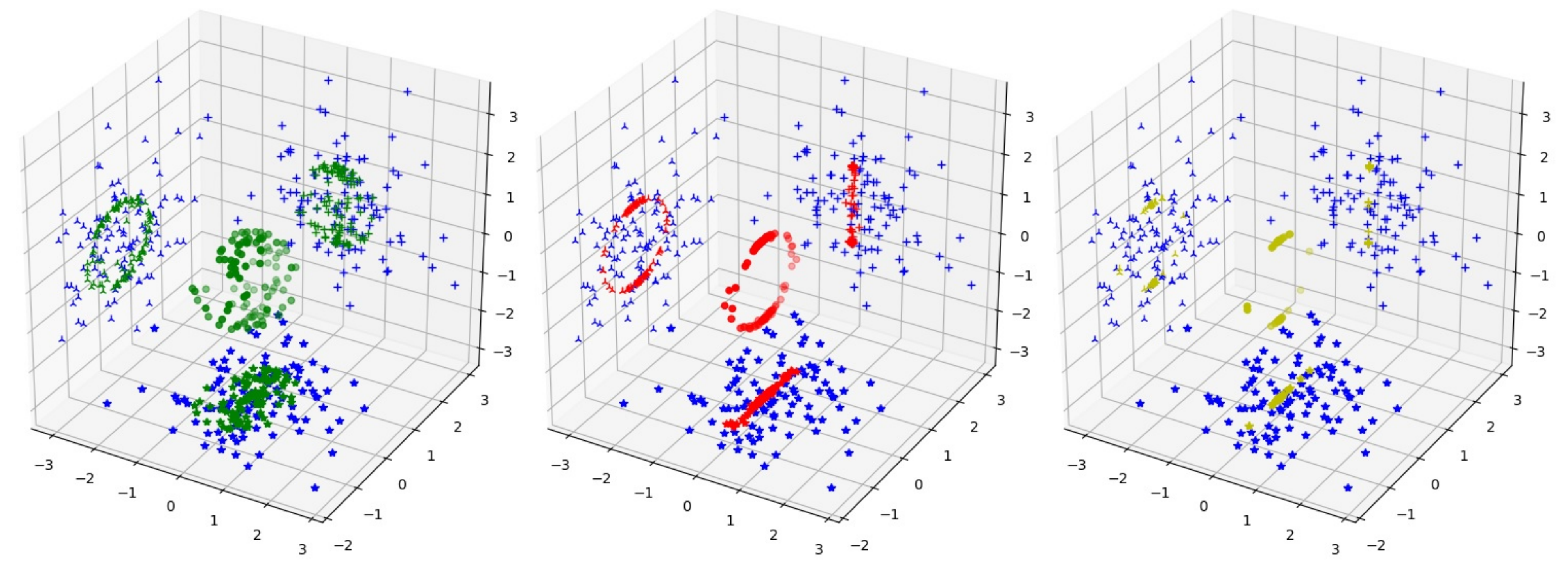}
    \caption{The normalized projection, i.e., $H_B\Omega$, $H_B^4\Omega$ and $H_B^{10}\Omega$ (LTR) of standard gaussian vectors $\Omega$ (blue points) where $H_B=\mathrm{diag}\{1,2,3\}$.}
    \vspace*{-10pt}
    \label{fig:power_iteration}
\end{figure}
With such an observation, we can calculate a better $U$ with the following steps:
\begin{enumerate}
    \item generate a \emph{standard Gaussian matrix} $\Omega$ and set $Y_0=\Omega$;
    \item iteratively use the extended Hessian-vector product to get $Y_{j}=[H_B(x)]^j \Omega=\Psi_B(x, Y_{j-1})$;
    \item calculate the basis $U$ via QR decomposition $Y_{q} = U R$.
\end{enumerate}

\begin{algorithm}
   \caption{\textbf{S}tochastic \textbf{P}rojected \textbf{A}pproximate \textbf{N}ewton}
   \label{alg:phantom}
\begin{algorithmic}[1]
   \STATE {\bfseries Input:} $F$, $x_0$, $T$, $m$, $l$, $q$, $\eta_t$
   \FOR{$t = 0$ to $T$}
        \STATE Select a uniformed sample batch $B \subseteq [N]$
        \STATE Generate a \emph{standard Gaussian matrix} $\Omega \in \mathbb{R}^{d\times l}$, and set $Y_0 = \Omega$
        \FOR{$j = 1$ to $2q+1$}
            \STATE $Y_{j} = \Psi_B(x_t,Y_{j-1})$
        \ENDFOR
        \STATE Compute the QR decomposition $Y_{2q+1} = UR$
        \STATE Set \begin{small}$Z = \Psi_B(x_t, U)$\end{small}, \begin{small}$\lambda_{\min,t} = \frac{1}{2}\sigma_{\min}(Z^TU)$\end{small} and \begin{small}$\lambda \le \min\left\{\sigma_{m+1,t}, \lambda_{\min,t}\right\}$\end{small}
        \STATE Recover the Hessian approximation inverse \begin{small}$\hat{H}_B^{-1}(x_t) = U\left(Z^TU\right)^{-1}U^T + \lambda^{-1}\left(I - UU^T\right)$
        \end{small}
        \STATE Calculate $x_{t+1} = x_t - \eta_t \hat{H}_B^{-1}(x_t)\nabla F(x_t)$
        \ENDFOR
        \STATE {\bfseries Return:} $x_{T+1}$
\end{algorithmic}
\end{algorithm}

\subsection{Details of \method}
\label{sec:overall_algorithm}
In this section, we show the complete algorithm about \method in Algorithm~\ref{alg:phantom}. Also, we explain the iteration complexity of \method, and compare it with the state-of-the-art optimizers.

In Algorithm~\ref{alg:phantom}, we use $\sigma_{i,t}$ as the abbreviation of $\sigma_i(H_B(x_t))$.
At each iteration, we first select a sample batch (Step~3) and generate some random matrix (Step~4).
With the random matrix $\Omega$, a proper $U$ can be found through the process we introduced in Section~\ref{sec:better_approxiamtion} (Step~5 to Step~8).
After that, we compute $H_B(x_t)U$ (Step~9) as an intermediate variable and select the perturbation constant $\lambda$.
Finally, we calculate the constructed Hessian inverse $\hat{H}_B^{-1}(x_t)$ (Step~10) like Eq.~\eqref{eq:perturb-hessian} and update the decision variables (Step~11).

\paragraph{Iteration Efficiency.} In order to illustrate the computational complexity of each iteration, we first introduce some condition numbers,
which are designed with respect to the component functions, i.e., $f_i(\cdot)$ and $f_B(\cdot)$. In such case, one typically assumes that each component is bounded by $\beta_{b,k}(x) \triangleq \max_B \sigma_k(H_B)$ and $\alpha_{b,k}(x) \triangleq \min_B \sigma_k(H_B)$ like LiSSA~\cite{agarwal2017second}, we define
\begin{equation*}
    \begin{split}
        \overline{\kappa}_{b,k} \triangleq \max_x \frac{\beta_{b,k}(x)}{\alpha_{N,d}(x)},\quad
        \hat{\kappa}_{b,k} \triangleq \frac{\max_x \beta_{b,k}(x)}{\min_x \alpha_{N,d}(x)},\quad
        \tilde{\kappa}_{b,k} \triangleq \max_{x}\frac{\beta_{b,k}(x)}{\alpha_{b,d}(x)}\quad \mathrm{and}\quad \dot{\kappa}_{b,k} \triangleq \max_{x, B}\frac{\sigma_{k}(H_B(x))}{\sigma_{d}(H_B(x))}.
    \end{split}
\end{equation*}
Such condition numbers have the following relations
\begin{equation}
    \label{eq:condition_number_rela}
    \begin{split}
        \overline{\kappa}_{b,k}&\le \hat{\kappa}_{b,k}\quad \mathrm{and}\quad \dot{\kappa}_{b,k}\le \tilde{\kappa}_{b,k}.\\
    \end{split}
\end{equation}

We compare per-iteration complexity among state-of-the-art optimizers, including \method, and list the results in Table~\ref{tab:complexity_comparisons} where we ignore $\log$ terms of $d$ and different $\kappa$s for brevity.
The iteration complexity of \method consists of three terms. The first term $\mathcal{O}(N d)$ represents the time complexity of the full gradient computation (Step~11).
The second term 
indicates the complexity of power iteration (Step~5 to Step~7). The calculation of $\Psi_B(\cdot)$ in Step~6 requires $\mathcal{O}(b l d)$ where the fact $q\propto \log d$ and $b = \Theta(\dot{\kappa}_{b,1}^2\dot{\kappa}_{b,m}^2\log d)$ will be detailedly demonstrated in our supplementary materials.
The last term $\mathcal{O}(l^2 d)$ shows the complexity of QR decomposition (Step~8).
In particular, the most time-consuming step in constructing $\hat{H}_B^{-1}(x_t)$ is to calculate $(Z^TU)^{-1}$ , which leads to an $\mathcal{O}(l^3)$ time complexity.
For a given $g(x_t)$, the computational cost of $\hat{H}_B^{-1}(x_t)g(x_t)$ is only $\mathcal{O}(l d)$  when the order of matrix multiplication is appropriately arranged.
Owing to the fact that $l\ll d$, such computational cost is much smaller than $\mathcal{O}(l^2 d)$ (the third term).
\newcommand{\tabincell}[2]{\begin{tabular}{@{}#1@{}}#2\end{tabular}}
\begin{table*}
    \centering
    \begin{tabular}{|c|c|}
    \hline
         Algorithm & Per-iteration Complexity \\
         \hline
        SVRG, SAGA, SDCA & $\mathcal{O}\left(N d+\left(\hat{\kappa}_{b,1}  d\right)\right)$ \\
        \hline
        NewSamp\cite{erdogdu2015convergence} & $\mathcal{O}\left(N d + \dot{\kappa}_{b,1}^2 d^2+m d^2\right)$\\
        \hline
        LiSSA\cite{agarwal2017second} & $\mathcal{O}\left(N d + \tilde{\kappa}_{b,1}^2\overline{\kappa}_{b,1} d^2 \right)$\\
        \hline
        \method (ours) & $\mathcal{O}\left(N d+ \dot{\kappa}^2_{b,1} \dot{\kappa}^2_{b,m} dl+l^2 d\right)$\\
    \hline
    \end{tabular}
    \vspace*{-5pt}
    \caption{Per-iteration complexity comparisons among stochastic first-order optimization methods, sub-sampled Newton Methods, LiSSA and ours. Notice that $\overline{\kappa}_{b,k}\le \hat{\kappa}$  and $\dot{\kappa}_{b,k}\le \tilde{\kappa}_{b,k}$.}
    \label{tab:complexity_comparisons}
    \vspace*{-15pt}
\end{table*}

Comparing with the NewSamp~\cite{erdogdu2015convergence} which updates decision variables with a sub-sampled constructed Hessian inverse, \method takes a significant acceleration at each iteration through the only first-order oracle and the Hessian-vector product requirements when the $m$-th singular value is close to the minimum singular value with $\dot{\kappa}^2_{b,m}\le \sqrt{d/m}$. In addition, per-iteration complexity in LiSSA~\cite{agarwal2017second} is even worse than NewSamp in general cases. However, in GLM models, LiSSA has better performance because of the fast calculation of Hessian-vector products. Even in GLM models, our SLAN can still be faster than LiSSA when the approximate Hessian is good enough, or the condition number $\overline{\kappa}_{b,1}$ is large with $\dot{\kappa}_{b,m}^2l\le \overline{\kappa}_{b,1}$. Notice that, the condition numbers of batch loss chosen, i.e., $\dot{\kappa}_{b,1}^2 \ge d$, will not be too large. Otherwise, per-iteration efficiency is governed by batch Hessian or Hessian-vector products calculation, which impairs the acceleration achieved by matrix approximation techniques.

\section{Theoretical Results}
\label{sec:analysis}
We will give a theoretical analysis of our results in this section.
We will first introduce some standard assumptions for the optimization problems, and then bound the error between our approximate Hessian and the sub-sampled one. Finally, we will show the local linear-quadratic convergence of our main algorithm (Algorithm~\ref{alg:phantom}), with the detailed coefficients of the convergence rate.
We further compare the convergence rate of our method with NewSamp~\cite{erdogdu2015convergence} and LiSSA~\cite{agarwal2017second}.
Due to space limitations, the details of proof arguments are provided in the supplementary materials.

\begin{assumption}(Hessian Lipschitz Continuity)
    \label{ass:1}
        For any subset $B\subset [N]$ and a second-order differentiable objective function $F$ like the Eq.\eqref{eq:obj_sum}, there exists a constant $M_b$ depending on $b$, such that $\forall x,\hat{x}\in \mathcal{D}$
        \begin{equation*}
            \begin{split}
            	\left\|H_B(x) - H_B(\hat{x})\right\|\le M_{b}\left\|x-\hat{x}\right\|.
            \end{split}
        \end{equation*}
\end{assumption}
\begin{assumption}(Gradient Lipschitz Continuity and Strong Convexity)
        \label{ass:2}
            For any subset $B\subset [N]$ and a
second-order differentiable function $F$ like the Eq.\eqref{eq:obj_sum}, there exists constants $\mu_{b}$ and $L_{b}$ depending on the $b$, such that for any $x,\hat{x}\in \mathcal{D}$
    \begin{equation*}
        \begin{split}
        \mu_{b}\left\|x-\hat{x}\right\| \le \left\|\nabla f_B(x) - \nabla f_B(\hat{x})\right\|\le L_{b}\left\|x-\hat{x}\right\|.
        \end{split}
    \end{equation*}
\end{assumption}
\begin{assumption}(Bounded Hessian)
    \label{ass:3}
    For any $i \in \{1,2,...,N\}$, and the Hessian of the function $f_i(x)$ in Eq.~\eqref{eq:obj_sum}, $\nabla^2 f_i(x)$ is upper bounded by an absolute constant $K$, i.e.,
    \begin{equation*}
        \begin{split}
        \max\limits_{i\le N}\left\|\nabla^2 f_i(x)\right\|\le K.
        \end{split}
    \end{equation*}
\end{assumption}

\begin{lemma}
    \label{thm:l2_norm_upper_bound}
    Suppose the Assumption~\ref{ass:1},~\ref{ass:2} and \ref{ass:3} hold. For every iteration $t$ in Algorithm~\ref{alg:phantom}, if the parameters satisfy: $m\le l- 4$, $\lambda \le \min\left\{\sigma_{m+1,t}, \frac{1}{2}\sigma_{\min}(Z^TU)\right\}$ and
    \begin{equation*}
        \begin{split}
            q\ge \Bigg\lceil \frac{1}{2}\log_{\frac{3}{2}} \left(34\sqrt{\frac{l}{l-m}}+\frac{16\sqrt{l}}{l-m+1}\cdot \sqrt{d-m}\right) \Bigg\rceil \propto \log d,
        \end{split}
    \end{equation*}
    with probability at least $1-6e^{m-l}$, we have
    \begin{equation*}
        \label{eqn:hessian_approximation_l2_norm_upperbound}
        \begin{split}
             & \left\|\hat{H}_B(x_t) - H_B(x_t)\right\| \le 3\sigma_{m+1,t} \triangleq \epsilon_H.
        \end{split}
    \end{equation*}
\end{lemma}
As far as we know, if we approximate the Hessian with an $m$-rank+sparse construction, the optimal Hessian approximation error will be $\sigma_{m+1,t} - \sigma_{d,t}$ which comes from NewSamp~\cite{erdogdu2015convergence}. Corresponding to such an approximation error, the construction complexity of NewSamp requires $\mathcal{O}(\dot{\kappa}_{b,1} d^2+m d^2)$ which can hardly be endured with a slightly larger instance dimension. While, from Lemma~\ref{thm:l2_norm_upper_bound}, \method improves the complexity by a factor at least $\mathcal{O}(d/(l \dot{\kappa}_{b,m}^2))$ ($l \ll d$) when the stochastic Hessian can be well approximated and keeps the approximation error nearly three times the optimal ($\sigma_{n,t}$ is a tiny constant). Such a guarantee cannot be promised for any quasi-Newton method. Although LiSSA~\cite{agarwal2017second} has a similar error bound for the approximate Hessian inverse, that approximation error is not comparable because it depends almost entirely on the concentration inequalities of the sub-sample processing but not the matrix approximation techniques. Additionally, we only bound the Hessian approximation error when $q = \mathcal{O}(\log d)$ and $\lambda \le \min\left\{\sigma_{m+1,t}, \frac{1}{2}\sigma_{\min}(Z^TU)\right\}$. In fact, such upper bounds on $q$ and $\lambda$ are possibly pessimistic and can likely be improved to a more average quantity. However, since the parameters $q = \mathcal{O}(1)$ and $\lambda = \frac{1}{2}\sigma_{m+1}(Z^TU)$ suffice for convergence in our experimental settings, we have not tried to optimize it further.

\begin{theorem}
    \label{thm:phantom_convergence_frame}
    Suppose $\lambda_{\min} \le 2\sigma_{m+1,t}$. Frame the hypotheses of Lemma.~\ref{thm:l2_norm_upper_bound}, if the parameters satisfy:
    \begin{equation*}
        \begin{split}
        \eta_t \le \frac{\sigma_{d,t}}{96\lambda_{\min,t} - 16\sigma_{d,t}}\  \mathrm{and}\  b = \Theta\left(K^2\sigma_{m+1,t}^2\sigma_{d,t}^{-4}\log(d)\right),
        \end{split}
    \end{equation*}
    with probability at least $1 - 6e^{m-l}$, we have
    \begin{equation*}
        \begin{split}
        \left\|x_{t+1} - x^* \right\|\le c_{1,t}\left\|x_t - x^*\right\| + c_{2,t}\left\|x_t - x^*\right\|^2.
        \end{split}
    \end{equation*}
    The coefficients $c_{1,t}$ and $c_{2,t}$ are
    \begin{equation*}
        \label{eq:phantom_convergence_constant}
        \begin{split}
            c_{1,t} =1-\frac{\sigma_{d,t}^2}{36\lambda_{\min,t}^2}\eta_t\quad \mathrm{and}\quad c_{2,t} =  \frac{M_{b}\eta_t}{\lambda_{\min,t}}.
        \end{split}
    \end{equation*}
\end{theorem}

\begin{remark*}
Notice that, we require $\lambda_{\min,t} \le 2\sigma_{m+1,t}$ in Theorem~\ref{thm:phantom_convergence_frame}, which is only introduced to simplify the formulation of coefficients, i.e., $c_{1,t}$ and $c_{2,t}$. For any $\lambda_{\min,t} \le \gamma \sigma_{m+1,t}\ (\gamma\ge 2)$, we can obtain a similar convergence rate by setting $\lambda = \gamma^{-1}\lambda_{\min,t}$. Thus, the theorem is without loss of generality.
\end{remark*}

Theorem~\ref{thm:phantom_convergence_frame} shows that \method has a similar composite convergence rate like NewSamp~\cite{erdogdu2015convergence} and LiSSA~\cite{agarwal2017second} where $c_1^t$ and $c_2^t$ are coefficients of the linear and quadratic terms, respectively.
Comparing the convergence with NewSamp whose first-order coefficient $c_{1,t}$ is abbreviated as $1-\alpha$, ours can be similar to $1-\alpha/6$.
In particular, truncated SVD is introduced in NewSamp to find an optimal $m$-rank approximate Hessian so that $c_1^t$ in NewSamp can be regarded as the optimal coefficient in most stochastic Newton method settings.
For LiSSA, its convergence rate is constrained by the initial decision variable $x_0$ and the sample size for calculating Hessian-vector product strictly, which causes that we cannot compare its convergence rate with ours directly.
However, \method has a better convergence wall-clock time which is validated in our experimental results (see Section~\ref{sec:experiment}).
In addition, during the analysis of the convergence rate, we notice that promoting the Hessian approximation error $\epsilon_H$ not only reduces the first-order coefficient in the linear-quadratic convergence rate  but also expands the range of step size.
This coincides with our intuition that  a faster convergence usually requires a smaller Hessian approximation error.
In summary, \method is designed to achieve a better trade-off between the iteration complexity and the convergence rate. With randomizing the process of constructing the special approximate Hessian, \method both accelerates the iteration and limits Hessian approximation error.

\section{Experiments}
\label{sec:experiment}
To evaluate the performances and fully demonstrate the advantages of our proposed method, we conduct our experiments on several machine learning tasks with different objective functions, including binary image classification and text classification. As we get similar conclusions on these two tasks, we only present the experimental result on the binary image classification in this section. We will further show all experimental details, including the parameters of all optimizers and the results on text classification tasks in the supplementary materials due to space limitations.

For the image classification tasks, we refer to the experimental settings in \cite{agarwal2017second} and \cite{erdogdu2015convergence}. We utilize the following objective function:
\begin{equation*}
        \min\limits_{x} -\frac{1}{N}\mathop{\sum}\limits_{i=1}^N\log\frac{1}{1 + \exp\left(-y_i\theta_i^Tx\right)} + \frac{1}{2}a \left\|x\right\|^2,
\end{equation*}
where $\theta_i \in R^d$ and $y_i \in \{-1, 1\}$ are the instances and labels, respectively, and $a$ is the $L_2$ regularization parameter which influences the condition number of the objective.

We choose state-of-the-art first- and second-order optimization methods as baselines, including SVRG~\cite{johnson2013accelerating}, NewSamp~\cite{erdogdu2015convergence},  S-LBFGS~\cite{moritz2016linearly}, SB-BFGS~\cite{gower2016stochastic}, and LiSSA~\cite{agarwal2017second}.  We implement all the methods in C++ and Intel Math Kernel Library(MKL). All the code for our experiments can be found in our supplementary material.
\begin{figure*}
    \centering
    \includegraphics[width=1.0\textwidth]{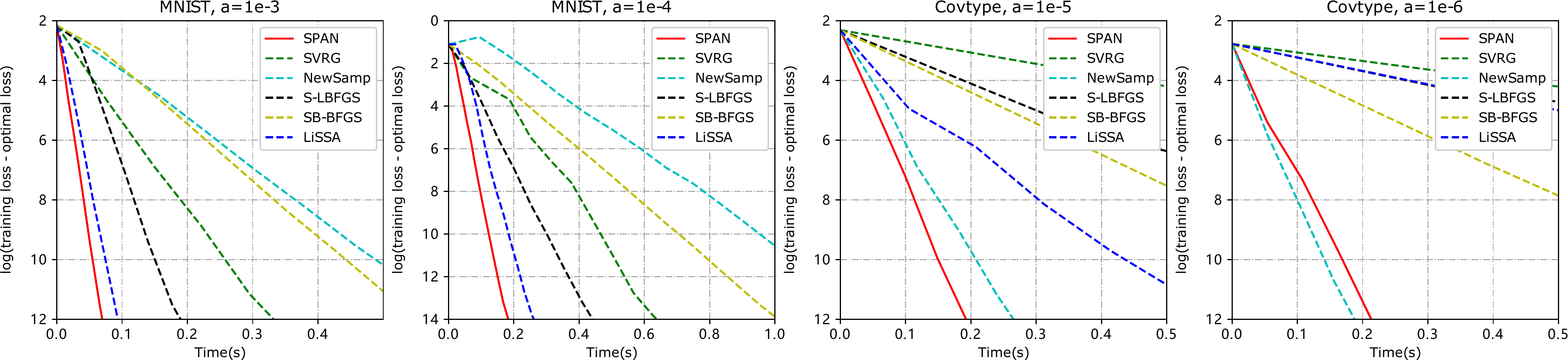}
    \caption{Training loss versus running time. The first two columns are for MNIST4-9 dataset. The last two are for CovType dataset. Note \method achieves the best or near the best performance with respect to wall-clock time.}
    \label{fig:wallclocktime}
\end{figure*}

\begin{figure*}
    \centering
    \includegraphics[width=1.0\textwidth]{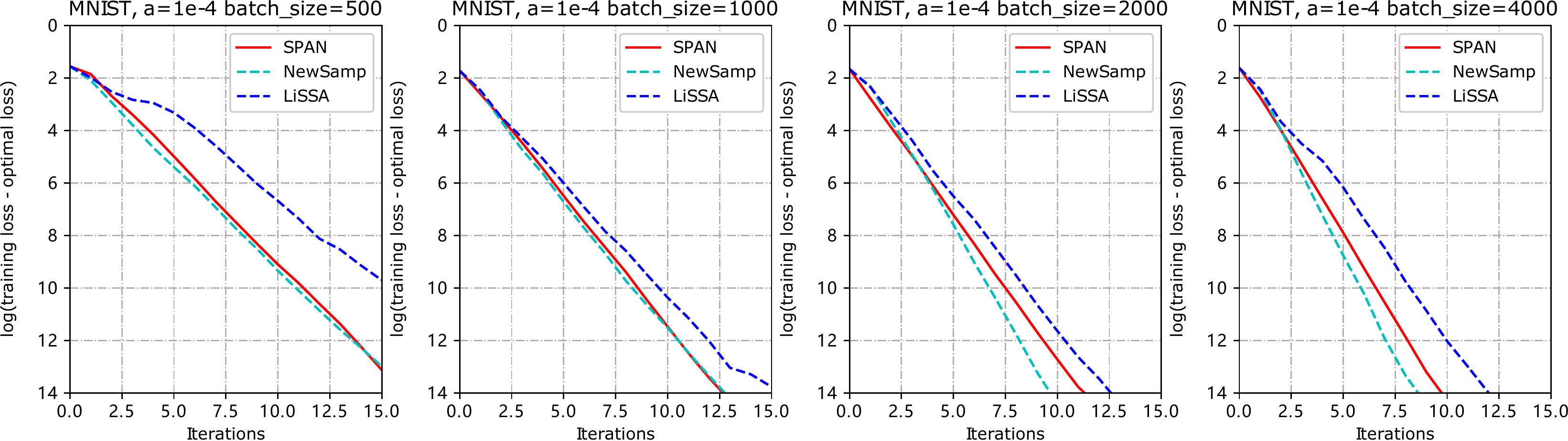}
    \caption{Empirical convergence rate comparison on MNIST4-9 dataset. Four columns represent different batch-size for sub-sampled Hessian.}
    \label{fig:iteration}
\end{figure*}

We adopt two datasets in our experiment, including the MNIST4-9 dataset~\cite{agarwal2017second} consisting of approximate $1.2\times 10^4$ instances, and CovType dataset~\cite{erdogdu2015convergence} consisting of approximate $5.0\times 10^5$ instances. For each dataset, we evaluate all the methods on two different condition numbers, respectively. Besides, for the fairness of comparison, in each set of experiments, we pick the optimal hyperparameters for every method and present the wall-clock time of all the methods in Figure~\ref{fig:wallclocktime}.

From Figure~\ref{fig:wallclocktime}, we can find that our method \method achieves the best results in almost all experimental settings. NewSamp gets a very close performance on the CovType dataset but consumes a lot of time on the MNIST dataset. LiSSA performs very close on the MNIST dataset but does not perform well on the CovType dataset. Meanwhile, there are significant gaps between \method and the other methods in all the experimental settings.

\begin{figure*}
    \centering
    \includegraphics[width=1.0\textwidth]{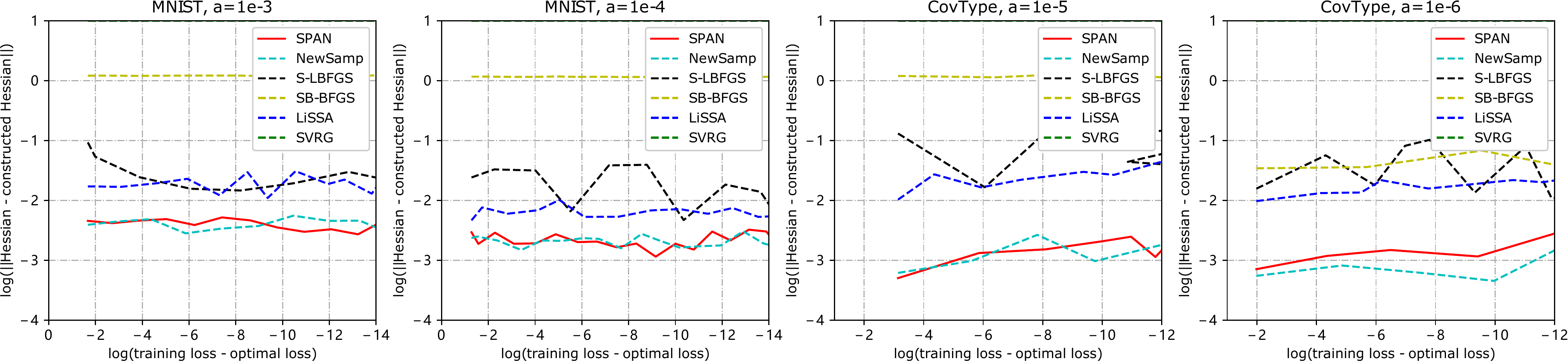}
    \caption{Hessian approximation error. Notice that \method achieves near lowest error.}
    \label{fig:hessian_approx}
\end{figure*}

To further explain the gaps between the second-order baselines and \method, we evaluate the empirical approximation error between the real full Hessian and the constructed approximate Hessian in Figure~\ref{fig:hessian_approx}.
We can see that the approximation error of our method is very close to that of the theoretical optimal solution (NewSamp) but much more accurate than the other methods.
The results of the experiment demonstrates that our constructed approximate Hessian in \method makes better use of second-order information than other stochastic second-order optimizers (S-LBFGS, SB-BFGS, and LiSSA).
Moreover, Figure~\ref{fig:iteration} illustrates the empirical convergence rate of these methods. As we can see, the empirical convergence rate of our proposed method is much better than LiSSA but slightly worse than NewSamp, which confirms our theoretical results and insights again.

Even though NewSamp guarantees an excellent convergence rate, the efficiency of the algorithm mainly depends on the number of parameters. NewSamp needs a long time to perform an iteration, even when there are only $784$ parameters in MNIST dataset. As a comparison, our proposed \method is much more competitive in terms of the robustness of the feature dimension and the per-iteration efficiency in various scenarios.

Considering the experiments mentioned above, \method has a tolerable Hessian approximation error which results in the runner up with respect to the empirical convergence rate. Besides, it enjoys the benefit of per-iteration efficiency, which makes it outperform others in practice. In summary, we conclude that \method achieves a better trade-off between per-iteration efficiency and convergence rate. Furthermore, it is robust for the sub-sampled batch size.

\section{Conclusion and future work}
\label{sec:conclu}
Newton and quasi-Newton methods converge at faster rates than gradient descent methods.
However, they are often expensive, computationally.
In this paper, we propose \method,  a novel method to optimize a smooth-strongly convex objective function.
\method utilizes the first-order oracle for Hessian approximation. Therefore, it is much faster than Newton method and its alike.
We give a theoretical analysis of its approximation accuracy and convergence rate.
Experiments on several real datasets demonstrate that our proposed method outperforms previous state-of-the-art methods.

For the future work, the constructed approximate Hessian of \method can be combined with advanced first-order methods, i.e., SVRG~\cite{johnson2013accelerating} and SAGA~\cite{defazio2014saga}, or help to introduce second-order information to complex objective functions, e.g., cubic regularization~\cite{nesterov2006cubic,kohler2017sub,tripuraneni2018stochastic}, neural networks with low complexity and competitive convergence rate.
\section*{Acknowledge}
Thanks the anonymous reviewers for their helpful comments and suggestions. We also thank Zhe Wang, Linyun Yu, Yitan Li, and Hao Zhou for their valuable discussions. 

\bibliographystyle{unsrt}  
\bibliography{references}  

\newpage
\appendix
\section{Some Important Lemmas}	In this section, we give several important definitions and lemmas which will be used in the proof of the convergence analysis (Section.~4 in this paper).

\begin{definition}(Orthogonal Projectors~\cite{halko2011finding})
    \label{def:orthogonal_projector_definition}
    An orthogonal projector is an Hermitian matrix $P$ that satisfies the polynomial $P^2 = P$. This identity implies $0\preceq P \preceq I$. An orthogonal projector is completely determined by its range(column space). For a given matrix $M$, we write $\mathcal{P}_M$ for the unique orthogonal projector with $\mathrm{range}\left(\mathcal{P}_M\right) = \mathrm{range}\left(M\right)$. When $M$ has full column rank, we can express this projector explicitly:
    \begin{equation*}
        \begin{split}
            \mathcal{P}_M = M\left(M^*M\right)^{-1}M^*,
        \end{split}
    \end{equation*}
    where $M^*$ denotes the conjugate transpose matrix of $M$. The orthogonal projector onto the complementary subspace, $\mathrm{range}\left(P\right)^{\perp}$, is $I - P$.
\end{definition}

\begin{lemma}(Matrix Concentration Inequalities~\cite{gross2010note})
    \label{lem:matrix_concentration_inequality_1}
    Let $\mathcal{X}$ be a finite set of Hermitian matrices in $\mathbb{R}^{d\times d}$ where $\forall X_i\in \mathcal{X}$, we have
    \begin{equation*}
        \mathbb{E}\left[X_i\right] = 0,\quad \left\|X_i\right\|\le \gamma,\quad \left\|\mathbb{E}\left[X_i^2\right]\right\|\le\sigma^2.
    \end{equation*}
    Given its size, let $S$ denote a uniformly random sample from $\left\{1,2,...,\left|\mathcal{X}\right|\right\}$ with or without replacement. Then we have
    \begin{equation*}
        \begin{split}
            \mathbb{P}\left(\left\|\frac{1}{|S|}\mathop{\sum}\limits_{i\in S}X_i\right\|\ge \epsilon\right)\le 2n\ \exp\left\{-\left|S\right|\min\left(\frac{\epsilon^2}{4\sigma^2},\frac{\epsilon}{2\gamma}\right)\right\}.
        \end{split}
    \end{equation*}
\end{lemma}

\begin{lemma}(Conjugation Rule)
    \label{lem:conjugation_rule}
    Suppose that $M\succeq 0$, for every real matrix $A$, the matrix $A^TMA\succeq 0$. In particular,
    \begin{equation*}
        M\preceq N \quad \Rightarrow \quad A^TMA \preceq A^TNA.
    \end{equation*}
\end{lemma}
\begin{proof}
    The positive semidefinite matrices form a convex cone, which induces a partial ordering on the linear space of Hermitian matrices: $M\preceq N$ if and only if $N-M$ is positive semidefinite, which means
    \begin{equation*}
        M\preceq N\quad \Rightarrow\quad 0\preceq N-M.
    \end{equation*}
    For any real vector $u$, $v = Au$ satisfies
    \begin{equation*}
        \begin{split}
            & 0\le v^T\left(N-M\right)v = u^TA^T\left(N-M\right)Au \\
            \Rightarrow & 0\le u^T\left(A^TNA - A^TMA\right)u \\
            \Rightarrow & A^TNA \preceq A^TMA.
        \end{split}
    \end{equation*}
\end{proof}

\begin{lemma}
    \label{lem:power_iteration_inequality}
    Let $P$ be an orthogonal projector, and let $M$ be a real square matrix. For each positive number q, we have
    \begin{equation*}
        \begin{split}
            \left\|PM\right\|\le\left\|P\left(MM^T\right)^qM\right\|^{1/(2q+1)}.
        \end{split}
    \end{equation*}
\end{lemma}
\begin{proof}
    Define the SVD of $M$ as $M = U \Sigma V^T$ where $U$ and $V$ are unitary matrices. Thus, we have $\Sigma = \mathrm{diag}\left\{\sigma_1, \sigma_2, \ldots, \sigma_{\mathrm{rank}(M)}, 0, 0, ..., 0\right\}$ where $\sigma_i$ denotes the $i$-th largest singular value of matrix $M$.
    Then, we have
    \begin{equation*}
        \begin{split}
            & \left\|PM\right\|^{2(2q+1)} = \left\|P M M^T P\right\|^{2q+1}  \mathop{=}^{\tiny\mathcircled{1}} \left\|U^T PU\Sigma^2U^T PU\right\|^{2q+1} \\
            \mathop{\le}^{\tiny\mathcircled{2}} & \left\|\left(U^T PU\right)\Sigma^{2(2q+1)} \left(U^T PU\right)\right\| \mathop{=}^{\tiny\mathcircled{3}} \left\|PU\Sigma^{2(2q+1)}U^T P\right\|\\
             = & \left\|P\left(MM^T\right)^{2q+1}P\right\| = \left\|P\left(MM^T\right)^q M\cdot M^T \left(MM^T\right)^q P\right\| = \left\|P\left(MM^T\right)^q M\right\|^2.
        \end{split}
    \end{equation*}
    where $\tiny\mathcircled{1}$ and $\tiny\mathcircled{3}$ use the unitary invariance of the spectral norm. The inequality $\tiny\mathcircled{2}$ is because
    \begin{equation*}
        \left\|\hat{P}X\hat{P}\right\|^t \le \left\|\hat{P}X^{t}\hat{P}\right\|,
    \end{equation*}
    where $\hat{P}$ is an orthogonal projector, $X$ is a nonnegative diagonal matrix and $t\ge 1$. This relation follows immediately from Theorem. \uppercase\expandafter{\romannumeral9}.2.10 in~\cite{bhatia2013matrix}.
\end{proof}

\begin{lemma}(Proposition 10.2 in~\cite{halko2011finding})
    \label{lem:4}
    Draw a $m\times l$ standard Gaussian matrix $G$ with $m\ge2$ and $l\ge m+2$, then
    \begin{equation*}
        \begin{split}
        & \left(\mathbb{E}\left\|G^{\dagger}\right\|_F^2\right)^{1/2} = \sqrt{\frac{m}{l-m-1}}\ and \quad \mathbb{E}\left\|G^{\dagger}\right\|\le\frac{e\sqrt{l}}{l-m}.
        \end{split}
    \end{equation*}
\end{lemma}

\begin{lemma}(Proposition 10.4 in~\cite{halko2011finding})
    \label{lem:3}
    Let $G$ be a $m\times l$ Gaussian matrix where $l\ge m+4$. For all $\theta\ge 1$,
    \begin{equation*}
        \begin{split}
            \mathbb{P}\left\{\left\|G^{\dagger}\right\|_F\ge\sqrt{\frac{12m}{l-m}}\cdot \theta\right\}\le4\theta^{m-l}\quad and\quad \mathbb{P}\left\{\left\|G^{\dagger}\right\|\ge\frac{e\sqrt{l}}{l-m+1}\cdot \theta\right\}\le \theta^{m-l-1}.
        \end{split}
    \end{equation*}
\end{lemma}

\begin{lemma}(Theorem. 4.5.7 in~\cite{bogachev1998gaussian}).
    \label{lem:concentration_for_gaussian_matrix_functions}
    Suppose that $h$ is a Lipschitz function on matrices:
    \begin{equation*}
        \begin{split}
            \left|h(X) - h(Y)\right|\le L\left\|X- Y\right\|_F\quad for\ all\ X,Y.
        \end{split}
    \end{equation*}
    Draw a standard Gaussian matrix $G$, then
    \begin{equation*}
        \begin{split}
            \mathbb{P}\left\{h(G)\ge \mathbb{E}\left[h(G)\right]+Lu\right\}\le e^{-u^2/2}.
        \end{split}
    \end{equation*}
\end{lemma}

\begin{lemma}
    \label{lem:6}
    Suppose $\mathrm{range}(N)\subseteq \mathrm{range}(M)$. Then, for each real matrix $A$, it holds that $\left\|\mathcal{P}_NA\right\|\le \left\|\mathcal{P}_MA\right\|$ and that $\left\|\left(I - \mathcal{P}_M\right)A\right\|\le\left\|\left(I - \mathcal{P}_N\right)A\right\|$.
\end{lemma}
\begin{proof}
    By applying Definition~\ref{def:orthogonal_projector_definition} and Lemma~\ref{lem:conjugation_rule}, we have $\mathcal{P}_{N}\preceq I$ and $\mathcal{P}_M\mathcal{P}_N\mathcal{P}_M \preceq \mathcal{P}_M$. Since $\mathrm{range}(N)\subset\mathrm{range}(M)$ implies $\mathcal{P}_M\mathcal{P}_N = \mathcal{P}_N$, we write
    \begin{equation*}
        \mathcal{P}_M\mathcal{P}_N\mathcal{P}_M
        = \mathcal{P}_N\mathcal{P}_M = \left(\mathcal{P}_M\mathcal{P}_N\right)^T = \mathcal{P}_N.    
    \end{equation*}
    That is to say, $\mathcal{P}_N\preceq \mathcal{P}_M$ provides $A^T\mathcal{P}_N A \preceq A^T\mathcal{P}_M A$. Then, we have
    \begin{equation*}
        \left\|\mathcal{P}_NA\right\|^2 = \left\|A^T \mathcal{P}_N A\right\| \le \left\|A^T \mathcal{P}_M A\right\| = \left\|\mathcal{P}_MA\right\|^2.
    \end{equation*}
    The second statement $\left\|\left(I - \mathcal{P}_M\right)A\right\|\le\left\|\left(I - \mathcal{P}_N\right)A\right\|$ follows from the first $\left\|\mathcal{P}_NA\right\|\le \left\|\mathcal{P}_MA\right\|$ by taking orthogonal complements.
\end{proof}

\begin{lemma}
    \label{lem:7}
    Suppose a real matrix satisfies $M\succeq 0$. Then
    \begin{equation*}
        I-\left(I+M\right)^{-1}\preceq M
    \end{equation*}
\end{lemma}
\begin{proof}
    Define $A = M^{1/2}$ as the positive semidefinite square root of $M$ promised by Theorem 7.2.6 in~\cite{horn1985matrix}, we have
    \begin{equation*}
        I - \left(I + A^2\right)^{-1} = \left(I+A^2\right)^{-1}A^2 = A\left(I+A^2\right)^{-1}A \preceq A^2 = M.
    \end{equation*}
    The first equality can be verified algebraically. 
    The second holds because rational functions of a diagonalizable matrix, such as $A$, commute. The last relation is because the conjugate rule $\left(I+A^2\right)^{-1}\preceq I$.
\end{proof}

\begin{lemma}(Proposition 8.3 in~\cite{halko2011finding})
    \label{lem:8}
    We have $\left\|M\right\|\le\left\|A\right\| + \left\|C\right\|$ for each partitioned positive semidefinite matrix
    \begin{equation*}
        M = \left[\begin{matrix}A & B \\ B^* & C\end{matrix}\right].
    \end{equation*}
\end{lemma}
\begin{lemma}
\label{lem:minimum_sigular_values_condition}
Let $M,\ N\in \mathbb{R}^{d\times d}$ be two positive definite matrices. If there is existing that
\begin{equation}
    \label{eq:lem_12_0}
    \left\|M - N\right\| \le \epsilon,
\end{equation}
the minimum singular values of $A$ and $B$, $\sigma_{\min}(M)$ and $\sigma_{\min}(N)$,  satisfy
\begin{equation}
    \label{eq:lem_12_1}
    \left|\sigma_{\min}(M) - \sigma_{\min}(N)\right| \le \epsilon.
\end{equation}
\end{lemma}
\begin{proof}
To make a contradiction with Eq.~\eqref{eq:lem_12_1}, we assume that the minimum singular values of $A$ and $B$ have
\begin{equation}
    \label{eq:lem_12_2}
    \sigma_{\min}(M) - \sigma_{\min}(N) > \epsilon.
\end{equation}
For the singular vectors $v_M$ and $v_N$ of $M$ and $N$ which respectively correspond to $\sigma_{\min}(M)$ and $\sigma_{\min}(N)$, we have
\begin{equation}
    \label{eq:lem_12_3}
    \begin{split}
        v_M^T M v_M  - v_N^T N v_N >\epsilon \mathop{\Rightarrow}^{\tiny\mathcircled{1}} v_N^T M v_N - v_N^T N v_N > \epsilon \mathop{\Rightarrow}^{\tiny\mathcircled{2}} v_N^T(M-N)v_N >\epsilon
    \end{split}
\end{equation}
where the inequality $\tiny\mathcircled{1}$ uses $v_N^TMv_N \ge v_M^TMv_M$. Since the inequality $\tiny\mathcircled{2}$ contradict with the given condition Eq.~\eqref{eq:lem_12_0}, the assumption Eq.~\eqref{eq:lem_12_1} will never be satisfied. Hence, we have
\begin{equation*}
    \sigma_{\min}(M) - \sigma_{\min}(N) \le \epsilon.
\end{equation*}
In addition, we can use a similar way to verify that
\begin{equation*}
    \sigma_{\min}(N) - \sigma_{\min}(M) \le \epsilon.
\end{equation*}
Hence, the proof is completed.
\end{proof}
\begin{lemma}
    \label{the:2}
    Let $A\in \mathbb{R}^{d\times d}$ be a positive definite matrix.  $A$ and a matrix $\tilde{A}$ which can be presented as
    \begin{equation*}
        \label{eq:lemma_10_formulation_1}
        \begin{split}
            & A = Q\Sigma Q^T = \left[\begin{matrix}Q_{\parallel} & Q_{\perp}\end{matrix}\right]\left[\begin{matrix}\Sigma_1 & \\ & \Sigma_2
            \end{matrix}\right]\left[\begin{matrix}Q_{\parallel}^T \\ Q_{\perp}^T\end{matrix}\right],\ \Sigma = \mathrm{diag}\left(\sigma_1,\ \sigma_2,\ ...,\ \sigma_d\right),\ \sigma_1\ge \sigma_2\ge ...\ge \sigma_{d}> 0,\\
            & \tilde{A} = Q\tilde{\Sigma} Q^T = \left[\begin{matrix}Q_{\parallel} & Q_{\perp}\end{matrix}\right]\left[\begin{matrix}\tilde{\Sigma}_1 & \\ & \tilde{\Sigma}_2
            \end{matrix}\right]\left[\begin{matrix}Q_{\parallel}^T \\ Q_{\perp}^T\end{matrix}\right],\ \tilde{\Sigma} = \mathrm{diag}\left(\tilde{\sigma}_1,\ \tilde{\sigma}_2,\ ...,\ \tilde{\sigma}_d\right),\ \left|\tilde{\sigma}_1\right|, \left|\tilde{\sigma}_2\right|, ..., \left|\tilde{\sigma}_{m}\right|> 0,
        \end{split}
    \end{equation*}
    where $Q$ is an unitary matrix and
    \begin{equation*}
        \begin{split}
            Q_{\parallel}\in \mathbb{R}^{d\times m},\quad Q_{\perp}\in \mathbb{R}^{d\times(d-m)},\quad \Sigma_1,\tilde{\Sigma}_1\in \mathbb{R}^{m\times m},\quad \Sigma_2,\tilde{\Sigma}_2\in \mathbb{R}^{(d-m)\times (d-m)}.
        \end{split}
    \end{equation*}
    Choose a test matrix $\Omega\in \mathbb{R}^{d\times l}$ ($m\le l$), and construct the sample matrix $Y = A\Omega$. Assuming that $Q_{\parallel}^T\Omega$ has full row rank, the orthogonal projector $\mathcal{P}_{Y}$ satisfies
    \begin{equation*}
        \left\|\left(I-\mathcal{P}_Y\right)\tilde{A}\right\|^2 \le \left\|\Sigma_2 Q_{\perp}^T\Omega\cdot \left(Q_{\parallel}^T\Omega\right)^{\dagger}\right\|^2\cdot\left\|\Sigma_1^{-1}\tilde{\Sigma}_1\right\|^2 + \left\|\tilde{\Sigma}_2\right\|^2,
    \end{equation*}
    where $\left(Q_{\parallel}^T\Omega\right)^{\dagger}$ means the Moore–Penrose inverse of matrix $Q_{\parallel}^T\Omega$.
\end{lemma}
\begin{proof}
    For the orthogonal projector $\mathcal{P}_{Y}$ and the unitary matrix $Q$, it is clear that $Q^T\mathcal{P}_{Y}Q$ is an orthogonal projector due to Definition~\ref{def:orthogonal_projector_definition}. Evidently,
    \begin{equation*}
        \mathrm{range}\left(Q^T\mathcal{P}_{Y}Q\right) = Q^T\mathrm{range}\left(Y\right) = \mathrm{range}\left(Q^TY\right).
    \end{equation*}
    Since ranges determine orthogonal projectors, we have $Q^T\mathcal{P}_{Y}Q = \mathcal{P}_{Q^TY}$ and the following equation
    \begin{equation}
        \begin{split}
            \label{eqn:lemma11_23}
            \left\|\left(I - \mathcal{P}_Y\right)\tilde{A}\right\| =  \left\|Q^T\left(I- \mathcal{P}_Y\right)Q\left[\begin{matrix}\tilde{\Sigma}_1 Q_{\parallel}^T \\ \tilde{\Sigma}_2 Q_{\perp}^T\end{matrix}\right]\right\| = \left\|\left(I - \mathcal{P}_{Q^TY}\right)\left[\begin{matrix}\tilde{\Sigma}_1 Q_{\parallel}^T \\ \tilde{\Sigma}_2 Q_{\perp}^T\end{matrix}\right]\right\|
        \end{split}
    \end{equation}
    where the first euqation is because the unitary invariance of the spectral norm. Subsequently, we construct a matrix $Z$ whose column space belongs to the range of $Q^TY$ to scale the right-most item of the Eq.\eqref{eqn:lemma11_23} as follows
    \begin{equation}
        \label{eqn:lemma11_24}
        \begin{split}
            Z &=Q^TY\left(Q_{\parallel}^T\Omega\right)^{\dagger}\Sigma_1^{-1}  = \left[\begin{matrix}\Sigma_1 Q_{\parallel}^T\Omega \\ \Sigma_2 Q_{\perp}^T\Omega\end{matrix}\right] \cdot \left(Q_{\parallel}^T\Omega\right)^{\dagger}\Sigma_1^{-1} \mathop{=}^{\tiny\mathcircled{1}} \left[\begin{matrix}I \\ F\end{matrix}\right],\\ 
            &where\ F= \Sigma_2 Q_{\perp}^T\Omega\cdot \left(Q_{\parallel}^T\Omega\right)^{\dagger}\Sigma_1^{-1}.
        \end{split}
    \end{equation}
    $\tiny\mathcircled{1}$ uses the full row rank properties of $Q_{\parallel}^T\Omega$, and the construction of $Z$ denotes that $\mathrm{range}(Z)\subset \mathrm{range}\left(Q^TY\right)$. Thus, Lemma~\ref{lem:6} implies
    \begin{equation}
        \label{eqn:orthogonal_relaxiation}
        \begin{split}
            \left\|\left(I - \mathcal{P}_{Q^TY}\right)\left[\begin{matrix}\tilde{\Sigma}_1 Q_{\parallel}^T \\ \tilde{\Sigma}_2 Q_{\perp}^T\end{matrix}\right]\right\| \le \left\|\left(I - \mathcal{P}_{Z}\right)\left[\begin{matrix}\tilde{\Sigma}_1 Q_{\parallel}^T \\ \tilde{\Sigma}_2 Q_{\perp}^T\end{matrix}\right]\right\|.
        \end{split}
    \end{equation}
    Squaring both sides of Eq.\eqref{eqn:orthogonal_relaxiation},  we obtain
    \begin{equation}
        \label{eqn:a8}
        \begin{split}
            & \left\|\left(I - \mathcal{P}_{Q^TY}\right)\left[\begin{matrix}\tilde{\Sigma}_1 Q_{\parallel}^T \\ \tilde{\Sigma}_2 Q_{\perp}^T\end{matrix}\right]\right\|^2 \le \left\|\left(I - \mathcal{P}_{Z}\right)\left[\begin{matrix}\tilde{\Sigma}_1 Q_{\parallel}^T \\ \tilde{\Sigma}_2 Q_{\perp}^T\end{matrix}\right] \right\|^2 \\
            = & \left\|\left[\begin{matrix}Q_{\parallel} & Q_{\perp}\end{matrix}\right]\left[\begin{matrix}\tilde{\Sigma}_1 & \\ & \tilde{\Sigma}_2
            \end{matrix}\right] \left(I - \mathcal{P}_{Z}\right) \left[\begin{matrix}\tilde{\Sigma}_1 & \\ & \tilde{\Sigma}_2
            \end{matrix}\right]\left[\begin{matrix}Q_{\parallel}^T \\ Q_{\perp}^T\end{matrix}\right] \right\| \le \left\|\left[\begin{matrix}\tilde{\Sigma}_1 & \\ & \tilde{\Sigma}_2
            \end{matrix}\right]\left(I - \mathcal{P}_{Z}\right)\left[\begin{matrix}\tilde{\Sigma}_1 & \\ & \tilde{\Sigma}_2
            \end{matrix}\right]\right\|,
        \end{split}
    \end{equation}
    where the matrix $Z$ has a full column rank because of Eq.~\eqref{eqn:lemma11_24}. Thus, we have
    \begin{equation*}
        \begin{split}
            & \mathcal{P}_{Z}=Z\left(Z^TZ\right)^{-1}Z^T = \left[\begin{matrix}I \\ F\end{matrix}\right]\left(I + F^TF\right)^{-1}\left[\begin{matrix}I & F^T\end{matrix}\right]\\
            \Rightarrow & I - \mathcal{P}_{Z} = \left[\begin{matrix}I - \left(I+F^TF\right)^{-1} & -\left(I+F^TF\right)^{-1}F^T\\ -F\left(I+F^TF\right)^{-1} & I-F\left(I+F^TF\right)^{-1}F^T \end{matrix}\right].
        \end{split}
    \end{equation*}
    Applying Lemma \ref{lem:7}, the top-left block verifies
    \begin{equation*}
        \begin{split}
            I-\left(I+F^TF\right)^{-1} \preceq F^TF.
        \end{split}
    \end{equation*}
    Besides, the bottom-right block satisfies
    \begin{equation*}
        \begin{split}
            I - F\left(I+F^TF\right)^{-1}F^T \preceq I.
        \end{split}
    \end{equation*}
    For convenience, we abbreviate $D = -\left(I+F^TF\right)^{-1}F^T$, the matrix $I-\mathcal{P}_{Z}$ satisfies
    \begin{equation*}
        \begin{split}
            & I-\mathcal{P}_{Z} \preceq \left[\begin{matrix} F^TF & D \\ D^T & I\end{matrix}\right]\\
            \Rightarrow & \left[\begin{matrix}\tilde{\Sigma}_1 & \\ & \tilde{\Sigma}_2
            \end{matrix}\right]\left(I - \mathcal{P}_{Z}\right)\left[\begin{matrix}\tilde{\Sigma}_1 & \\ & \tilde{\Sigma}_2
            \end{matrix}\right] \preceq \left[\begin{matrix}\tilde{\Sigma}_1F^TF\tilde{\Sigma}_1 & \tilde{\Sigma}_1D\tilde{\Sigma}_2 \\ \tilde{\Sigma}_2D^T\tilde{\Sigma}_1 & \tilde{\Sigma}_2^2\end{matrix}\right].
        \end{split}
    \end{equation*}
    Subsequently, Lemma \ref{lem:8} provides the $l_2$ norm upper bound of the above matrix as
    \begin{equation*}
        \begin{split}
            \left\|\left[\begin{matrix}\tilde{\Sigma}_1 & \\ & \tilde{\Sigma}_2
            \end{matrix}\right]\left(I - \mathcal{P}_{Z}\right)\left[\begin{matrix}\tilde{\Sigma}_1 & \\ & \tilde{\Sigma}_2
            \end{matrix}\right]\right\| \le \left\|\tilde{\Sigma}_1F^TF\tilde{\Sigma}_1\right\|+\left\|\tilde{\Sigma}_2^2\right\| = \left\|F\tilde{\Sigma}_1\right\|^2 + \left\|\tilde{\Sigma}_2\right\|^2.
        \end{split}
    \end{equation*}
    As our construction in Eq.~\eqref{eqn:lemma11_24}, we have $F = \Sigma_2 Q_{\perp}^T\Omega\cdot \left(Q_{\parallel}^T\Omega\right)^{\dagger}\Sigma_1^{-1}$. Thus, we have
    \begin{equation*}
        \begin{split}
            \left\|\left[\begin{matrix}\tilde{\Sigma}_1 & \\ & \tilde{\Sigma}_2
            \end{matrix}\right]\left(I - \mathcal{P}_{Z}\right)\left[\begin{matrix}\tilde{\Sigma}_1 & \\ & \tilde{\Sigma}_2
            \end{matrix}\right]\right\| \le \left\|\Sigma_2 Q_{\perp}^T\Omega\cdot \left(Q_{\parallel}^T\Omega\right)^{\dagger}\right\|^2\cdot\left\|\Sigma_1^{-1}\tilde{\Sigma}_1\right\|^2 + \left\|\tilde{\Sigma}_2\right\|^2
        \end{split}
        \label{eqn:a9}
    \end{equation*}
    Finally, after introducing the inequality into Eq.~\eqref{eqn:a8}, the proof is completed.
\end{proof}

\section{Proofs in Section 4}
\subsection{Proof of Lemma 4.1}
\label{lem:4_1}
\begin{proof}
    According to the construction of $Y_{2q+1}$ (Step.~5 to Step.~7 in Algorithm 1) and $Z$ (Step.~9 in Algorithm 1), we have
    \begin{equation*}
        \begin{split}
            Y_{2q+1} = \left(\nabla^2 f_{B}(x_t)\right)^{2q+1}\Omega\ \in \mathbb{R}^{d\times l },\quad Z^TU = U^T\nabla^2 f_{B}(x_t)U,
        \end{split}
    \end{equation*}
    where $U$ and $Y_{2q+1}$ satisfy 
    \begin{equation*}
        \begin{split}
            U^TU = I,\quad Y_{2q+1} = UR,\quad \mathrm{range}(Y) = \mathrm{range}(U),\quad \mathcal{P}_{Y} =Y(Y^TY)^{-1}Y^T =  UU^T.
        \end{split}
    \end{equation*}
    Here, we denote the \emph{orthogonal projector} of $Y$ as $\mathcal{P}_{Y}=UU^T$. With the approximate Hessian $\hat{H}_B^{-1}(x_t)$ (Eq.~(6) in Section.~2.2), we have the Hessian approximation error satisfies 
    \begin{equation}
        \begin{split}
            & \left\|\hat{H}_B(x_t) - \nabla^2 f_{B}(x_t)\right\| = \left\|U\left(U^T\nabla^2 f_{B}(x_t)U\right)U^T + \lambda \left(I -UU^T\right) - \nabla^2 f_{B}(x_t)\right\|\\
            = & \left\|U\left(U^T\nabla^2 f_{B}(x_t)U\right)U^T - UU^T\nabla^2 f_{B}(x_t) + \lambda UU^T\left(I- UU^T\right) \right.\\
            & \left.+ UU^T\nabla^2 f_{B}(x_t) + \lambda \left(I- UU^T\right)-\nabla^2 f_{B}(x_t)\right\|\\
            \le & \left\|UU^T\left(\nabla^2 f_{B}(x_t) - \lambda I\right)\left(UU^T - I\right)\right\| + \left\|\left(UU^T - I\right)\left(\nabla^2 f_{B}(x_t) - \lambda I\right)\right\|\\
            \le & 2\left\|\left(UU^T - I\right)\left(\nabla^2 f_{B}(x_t) - \lambda I\right)\right\|.
        \end{split}
        \label{eqn:approximation_upper_bound}
    \end{equation}
     With algorithm settings and positive definite properties of $\nabla^2 f_{B}(x_t)$, we reformulate the $Y$ and introduce an auxiliary variable $\tilde{A}$ in the following
    \begin{equation}
        \begin{split}
            Y_{2q+1} & = \left(\nabla^2 f_{B}(x_t)\nabla^2 f_{B}(x_t)^T\right)^q\nabla^2 f_{B}(x_t)\Omega \\
            &=  Q\Sigma^{2q+1}Q^T\Omega=\left[\begin{matrix}Q_{\parallel} & Q_{\perp}\end{matrix}\right]\left[\begin{matrix}\Sigma_1^{2q+1} & \\ & \Sigma_2^{2q+1}
            \end{matrix}\right]\left[\begin{matrix}P_{\parallel}^T \\ P_{\perp}^T\end{matrix}\right]\Omega \\
            \tilde{A} & = \left(\left(\nabla^2 f_{B}(x_t) - \lambda I\right)\cdot\left(\nabla^2 f_{B}(x_t) - \lambda I\right)^T\right)^q\cdot\left(\nabla^2 f_{B}(x_t) - \lambda I\right)\\
            &= Q\tilde{\Sigma}^{2q+1}Q^T=\left[\begin{matrix}Q_{\parallel} & Q_{\perp}\end{matrix}\right]\left[\begin{matrix}\tilde{\Sigma}_1^{2q+1} & \\ & \tilde{\Sigma}_2^{2q+1}
            \end{matrix}\right]\left[\begin{matrix}Q_{\parallel}^T \\ Q_{\perp}^T\end{matrix}\right],
        \end{split}
        \label{eqn:power_scheme_reduction}
    \end{equation}
    where we set the SVD of the batch Hessian and diagonal elements of $\Sigma$ as
    \begin{equation*}
        \begin{split}
            & \nabla^2 f_{B}(x_t) = Q\Sigma Q^T,\quad Q_{\parallel}\in\mathbb{R}^{d\times m}, \quad Q_{\perp}\in\mathbb{R}^{d\times(d-m)}\quad and\quad \Sigma = \mathrm{diag}\left\{\sigma_1^t, \sigma_2^t, ..., \sigma_{d}^t\right\}.
        \end{split}
    \end{equation*}
    Thus, we have the following inequality according to Lemma \ref{lem:power_iteration_inequality}
    \begin{equation}
        \label{eqn:actual_error_bound}
        \begin{split}
            \left\|\left(I - UU^T\right)\cdot\left(\nabla^2 f_{B}(x_t) - \lambda I\right)\right\| 
            \le \left\|\left(I - UU^T\right)\cdot \tilde{A} \right\|^{\frac{1}{2q+1}} = \underbrace{\left\|\left(I - \mathcal{P}_Y\right)\cdot \tilde{A} \right\|^{\frac{1}{2q+1}}}_{part\ 1}.
        \end{split}
    \end{equation}
    Then, we introduce Lemma \ref{the:2} to bound the $part\ 1$ of Eq.\eqref{eqn:actual_error_bound} as
    \begin{equation}
        \begin{split}
            & \left\|\left(I - \mathcal{P}_{Y}\right)\tilde{A}\right\|^2\le \left\|\Sigma_2^{2q+1}Q_{\perp}^T\Omega\cdot \left(Q_{\parallel}^T\Omega\right)^{\dagger}\right\|^2 \cdot \left\|\Sigma_1^{-1}\tilde{\Sigma}_1\right\|^{4q+2} + \left\|\tilde{\Sigma}_2\right\|^{4q+2}\\
            \Rightarrow & \left\|\left(I - \mathcal{P}_{Y}\right) \tilde{A}\right\| \le \mathrm{max}\left\{\left|\frac{\sigma_1^t-\lambda}{\sigma_1^t}\right|, \left|\frac{\sigma_m^t-\lambda}{\sigma_m^t}\right|\right\}^{2q+1}\left\|\Sigma_2^{2q+1}Q_{\perp}^T\Omega\cdot \left(Q_{\parallel}^T\Omega\right)^{\dagger}\right\|+\left\|\tilde{\Sigma}_2\right\|^{2q+1}.
        \end{split}
        \label{eqn:a22}
    \end{equation}
    Subsequently, we can introduce a test matrix $\Omega$ with some special properties to scale Eq.\eqref{eqn:a22}. We set $\Omega\in \mathbb{R}^{d\times l}$ as a standard Gaussian matrix, and $m\le l - 4$ where $l \ll d$. As a result, the matrix $\left[Q_{\parallel}^T\Omega,\  Q_{\perp}^T\Omega\right]$ is also a standard Gaussian matrix because of the unitary of $Q$ and the rotational invariance of $\Omega$. Besides, $Q_{\parallel}^T\Omega$ and $Q_{\perp}^T\Omega$ are nonoverlapping submatrices of $\left[Q_{\parallel}^T\Omega,\  Q_{\perp}^T\Omega\right]$, and these two matrices are stochastically independent. Hence, we can learn how the error depends on the matrix $Q_{\perp}^T\Omega$ by conditioning on the event that $Q_{\parallel}^T\Omega$ is relatively regular. We set a \textbf{event} as
    \begin{equation}
        \label{eqn:a11}
        \begin{split}
            E_{\theta} = \left\{P_{\parallel}^T\Omega:\ \left\|\left(P_{\parallel}^T\Omega\right)^{\dagger}\right\|\le \frac{e\sqrt{l}}{l-m+1}\cdot \theta\quad and\quad \left\|\left(P_{\parallel}^T\Omega\right)^{\dagger}\right\|_F\le \sqrt{\frac{12m}{l-m}}\cdot \theta \right\}.
        \end{split}
    \end{equation}
    According to Lemma \ref{lem:3}, we find
    \begin{equation}
        \label{eqn:a12}
        \begin{split}
        \mathbb{P}\left(E_{\theta}^c\right) = 1 - \mathbb{P}\left(E_{\theta}\right) \le \theta^{m-l-1}+4\theta^{m-l}\le 5\theta^{m-l}.
        \end{split}
    \end{equation}
    Consider the function $h(X) = \left\|\Sigma_2^{2q+1}X \left(Q_{\parallel}^T\Omega\right)^{\dagger}\right\|$ whose Lipschitz constant $L$ can be calculated in the following
    \begin{equation*}
        \begin{split}
        & \left|h(X) - h(Y)\right|\le \left\|\Sigma_2^{2q+1}\left(X-Y\right) \left(Q_{\parallel}^T\Omega\right)^{\dagger}\right\| \\
        \le & \left\|\Sigma_2^{2q+1}\right\|\left\|X-Y\right\|\left\|\left(Q_{\parallel}^T\Omega\right)^{\dagger}\right\| \le \left\|\Sigma_2^{2q+1}\right\|\left\|X-Y\right\|_F\left\|\left(Q_{\parallel}^T\Omega\right)^{\dagger}\right\|.
        \end{split}
    \end{equation*}
    That is to say, $L\le \left\|\Sigma_2^{2q+1}\right\|\left\|\left(Q_{\parallel}^T\Omega\right)^{\dagger}\right\| $. With Lemma \ref{lem:4}, we have
    \begin{equation*}
        \begin{split}
            \mathbb{E}\left[h\left(Q_{\perp}^T\Omega\right) \big| Q_{\parallel}^T\Omega\right] \le \left\|\Sigma_2^{2q+1}\right\|\left\|\left( Q_{\parallel}^T\Omega\right)^{\dagger}\right\|_F + \left\|\Sigma_2^{2q+1}\right\|_F\left\|\left( Q_{\parallel}^T\Omega\right)^{\dagger}\right\|.
        \end{split}
    \end{equation*}
    Applying the concentration measure inequality, Lemma \ref{lem:concentration_for_gaussian_matrix_functions}, conditionally to the random variable $h\left(Q_{\perp}^T\Omega\right)$ results in
    \begin{equation*}
        \begin{split}
            \mathbb{P}\left\{\left\|\Sigma_2^{2q+1}Q_{\perp}^T\Omega\cdot \left(Q_{\parallel}^T\Omega\right)^{\dagger}\right\| > \left\|\Sigma_2^{2q+1}\right\|\left\|\left( Q_{\parallel}^T\Omega\right)^{\dagger}\right\|_F + \left\|\Sigma_2^{2q+1}\right\|_F\left\|\left( Q_{\parallel}^T\Omega\right)^{\dagger}\right\|+\right.\\
            \left. \left\|\Sigma_2^{2q+1}\right\|\left\|\left( Q_{\parallel}^T\Omega\right)^{\dagger}\right\|\cdot u\theta \Big| E_{\theta}\right\} \le e^{\frac{-u^2}{2}}.
        \end{split}
    \end{equation*}
    According to Eq.~\eqref{eqn:a11}, we have a explicit bound under the event $E_t$ as follows
    \begin{equation*}
        \begin{split}
            \mathbb{P}\left\{\left\|\Sigma_2^{2q+1}Q_{\perp}^T\Omega\cdot \left(Q_{\parallel}^T\Omega\right)^{\dagger}\right\| > \left\|\Sigma_2^{2q+1}\right\|\sqrt{\frac{12m}{l-m}}\cdot \theta + \left\|\Sigma_2^{2q+1}\right\|_F\frac{e\sqrt{l}}{l-m+1}\cdot \theta+\right.\\
            \left. \left\|\Sigma_2^{2q+1}\right\|\frac{e\sqrt{l}}{l-m+1}\cdot u\theta \Big| E_{\theta}\right\} \le e^{\frac{-u^2}{2}}.
        \end{split}
    \end{equation*}
    With the fact Eq.\eqref{eqn:a12}, we can remove the conditioning and obtain
    \begin{equation}
        \begin{split}
        \mathbb{P}\left\{\left\|\Sigma_2^{2q+1}Q_{\perp}^T\Omega\cdot \left(Q_{\parallel}^T\Omega\right)^{\dagger}\right\| > \left\|\Sigma_2^{2q+1}\right\|\sqrt{\frac{12m}{l-m}}\cdot \theta + \left\|\Sigma_2^{2q+1}\right\|_F\frac{e\sqrt{l}}{l-m+1}\cdot \theta+\right.\\
        \left. \left\|\Sigma_2^{2q+1}\right\|\frac{e\sqrt{l}}{l-m+1}\cdot u\theta \right\} \le 5\theta^{m-l} + e^{\frac{-u^2}{2}}.
        \end{split}
        \label{eqn:a13}
    \end{equation}
    Because of the SVDs presented in Eq.~\eqref{eqn:power_scheme_reduction}, we have
    \begin{equation}
        \begin{split}
        \left\|\Sigma_2^{2q+1}\right\| = \left(\sigma_{m+1}^t\right)^{2q+1},\quad \left\|\Sigma_2^{2q+1}\right\|_F \le \sqrt{d-m}\left\|\Sigma_2^{2q+1}\right\|.
        \end{split}
        \label{eqn:a16}
    \end{equation}
    After introducing Eq.\eqref{eqn:a13} and Eq.\eqref{eqn:a16} to Eq.\eqref{eqn:a22}, we have
    \begin{equation}
        \label{eqn:a19}
        \begin{split}
        \left\|\left(I - \mathcal{P}_{Y}\right)\tilde{A}\right\| \le & \left[\mathrm{max}\left\{\left|\frac{\sigma_1^t - \lambda}{\sigma_1^t}\right|, \left|\frac{\sigma_m^t - \lambda}{\sigma_m^t}\right|\right\}^{2q+1}\left(\sqrt{\frac{12m}{l-m}}\cdot \theta\right.\right.\\
        & \left.\left.+\frac{e\sqrt{l}}{l-m+1}\cdot \theta\sqrt{d-m}+\frac{e\sqrt{l}}{l-m+1}\cdot u\theta\right) \right] \cdot \left(\sigma_{m+1}^t\right)^{2q+1}\\
        & +\left[\mathrm{max}\left\{\left|\sigma_{m+1}^t - \lambda\right|, \left|\sigma_{d}^t - \lambda\right|\right\}\right]^{2q+1}.
        \end{split}
    \end{equation}
    with probability at least $1 - 5\theta^{m-l}-e^{\frac{-u^2}{2}}$. For making the upper bound clear, we set $\theta=e$, $u=\sqrt{2(l-m)}$, and have
    \begin{equation*}
        \begin{split}
            & e\sqrt{\frac{12m}{l-m}} + \frac{e^2\sqrt{2(l-m)l}}{l-m+1} + \frac{e^2\sqrt{(d-m)l}}{l-m+1}\\
            \le & e\sqrt{\frac{12m}{l-m}}+e^2\sqrt{\frac{2l}{l-m}} + \frac{e^2\sqrt{(d-m)l}}{l-m+1} \\
            \le & 17\sqrt{1+\frac{m}{l-m}} + \frac{8\sqrt{(d-m)l}}{l-m+1}\\
            =& 17\sqrt{\frac{l}{l-m}} + \frac{8\sqrt{(d-m)l}}{l-m+1}.
        \end{split}
    \end{equation*}
    Hence, we make some abbreviations as 
    \begin{equation*}
        \begin{split}
             \gamma = \mathrm{max}\left\{\left|\frac{\sigma_1^t - \lambda}{\sigma_1^t}\right|, \left|\frac{\sigma_m^t - \lambda}{\sigma_m^t}\right|\right\},\quad c = \left(17\sqrt{\frac{l}{l-m}}+\frac{8\sqrt{(d-m)l}}{l-m+1}\right)^{\frac{1}{2q+1}},  \\
        \end{split}
    \end{equation*}
    and Eq.\eqref{eqn:a19} can be reformulated and scaled as follows
    \begin{equation}
        \label{eqn:a20}
        \begin{split}
            \left\|\left(I - \mathcal{P}_{Y}\right)\tilde{A}\right\| \le &  \left[\left(\gamma c\right)^{2q+1}\left(\sigma_{m+1}^t\right)^{2q+1}\right]+\left[\mathrm{max}\left\{\left|\sigma_{m+1}^t - \lambda\right|, \left|\sigma_{d}^t - \lambda\right|\right\}\right]^{2q+1}\\
            \le & 2\mathrm{max}\left\{\gamma c \sigma_{m+1}^t, \left|\sigma_{m+1}^t - \lambda\right|, \left|\sigma_{d}^t - \lambda\right|\right\}^{2q+1}.
        \end{split}
    \end{equation}
    Finally, we introduce Eq.\eqref{eqn:a20} and Eq.\eqref{eqn:actual_error_bound} to Eq.\eqref{eqn:approximation_upper_bound} to have the following error upper bound with probability at least $1-6e^{m-l}\ (l\ge m+4)$ 
    \begin{equation*}
        \begin{split}
            & \left\|\hat{H}_B(x_t) - \nabla^2 f_{B_H}(x_t)\right\| \le 2^{\frac{2q+2}{2q+1}}\mathrm{max}\left\{\gamma c \sigma_{m+1}^t,\left|\sigma_{m+1}^t - \lambda\right|, \left|\sigma_d^t - \lambda\right|\right\} .
        \end{split}
    \end{equation*}
    Due to the fact $\lambda\le \min\{\sigma_{m+1}^t, \sigma_{\min}(U^T\nabla f_{B_H}(x_t)U)\}$, we have $\gamma< 1$ and
    \begin{equation*}
        \left\|\hat{H}_B(x_t) - \nabla^2 f_{B_H}(x_t)\right\| \le 2\left(34\sqrt{\frac{l}{l-m}}+\frac{16\sqrt{(d-m)l}}{l-m+1}\right)^{\frac{1}{2q+1}}\sigma_{m+1}^t.
    \end{equation*}
    That is to say, when the power iteration $q$ satisfy
    \begin{equation*}
        q\ge \Bigg\lceil \frac{1}{2}\log_{\frac{3}{2}} \left(34\sqrt{\frac{l}{l-m}}+\frac{16\sqrt{l}}{l-m+1}\cdot \sqrt{d-m}\right) \Bigg\rceil \propto \log d,
    \end{equation*}
    the Hessian approximation error has an upper bound as
    \begin{equation*}
        \left\|\hat{H}_B(x_t) - \nabla^2 f_{B_H}(x_t)\right\| \le 3\sigma_{m+1}^t.
    \end{equation*}
    
\end{proof}

\begin{corollary}
    \label{coro:ori_rlra_bound}
    Frame the hypotheses of Lemma~\ref{the:2}, and assume $l-m\ge 4$. Then
    \begin{equation}
        \left\|\left(I - \mathcal{P}_Y\right)A\right\|\le \left(1 + 17\sqrt{\frac{l}{l-m}}+\frac{8\sqrt{l}}{l-m+1}\cdot \sqrt{d-m}\right)^{\frac{1}{2q+1}}\sigma_{m+1},
    \end{equation}
    with failure probability at most $6e^{m-l}$.
\end{corollary}

\subsection{Proof of Theorem 4.2}
\label{thm:4_2}
\begin{proof}
    According to Step.~11 in \method Algorithm, we have
    \begin{equation}
        \label{eqn:thm42_eq1}
        \begin{split}
            \left\|x_{t+1} - x^*\right\|  & = \left\|x_t - \eta_t \hat{H}_B^{-1}(x_t)\nabla F(x_t) - x^*  \right\|\\
            &  = \left\|x_t - x^* - \eta_t \hat{H}_B^{-1}(x_t) \int_0^1{\nabla^2 F(x^*+ \tau(x_t - x^*))(x_t - x^*)d\tau}\right\|\\
            & \le \left\|x_t - x^*\right\| \cdot \underbrace{\left\|I - \eta_t \hat{H}_B^{-1}(x_t)\int_0^1{\nabla^2 F(x^*+ \tau(x_t - x^*))d\tau}\right\|}_{part\ 1}.
        \end{split}
    \end{equation}
    Then, the \emph{part 1} of Eq.\eqref{eqn:thm42_eq1} can be scaled as
    \begin{equation}
        \label{eqn:thm42_eq2}
        \begin{split}
            & \left\|I - \eta_t \hat{H}_B^{-1}(x_t)\int_0^1{\nabla^2 F(x^*+ \tau(x_t - x^*))d\tau}\right\|\\
            \le & \left\|I - \eta_t\hat{H}_B^{-1}(x_t)\nabla^2 f_{B}(x_t) + \eta_t\hat{H}_B^{-1}(x_t)\nabla^2 f_{B}(x_t) - \eta_t\hat{H}_B^{-1}(x_t)\nabla^2 F(x_t) + \eta_t\hat{H}_B^{-1}(x_t)\nabla^2 F(x_t) \right.\\
            & \left. - \eta_t \hat{H}_B^{-1}(x_t)\int_0^1{\nabla^2 F(x^*+ \tau(x_t - x^*))d\tau}\right\| \\
            \le & \underbrace{\left\|I - \eta_t \hat{H}_B^{-1}(x_t)\nabla^2 f_{B_H}(x_t)\right\|}_{part\ 1}+\eta_t \left\|\hat{H}_B^{-1}(x_t)\right\| \cdot \left(\underbrace{\left\|\nabla^2 f_{B}(x_t) - \nabla^2 F(x_t)\right\|}_{part\ 2} \right.\\
            & \left.+ \underbrace{\left\|\nabla^2 F(x_t) - \int_0^1{\nabla^2 F(x^* + \tau (x_t - x^*))d\tau}\right\|}_{part\ 3}\right).
        \end{split}
    \end{equation}
    
    \textbf{For the $part\ 1$ of Eq.\eqref{eqn:thm42_eq2}:} We have known that
    \begin{equation}
        \label{eq:part_1_relaxiation}
        \begin{split}
            & \left\|I - \eta_t\hat{H}_{B}^{-1}(x_t)\nabla^2f_B(x_t)\right\|^2\\
            =&\max_{\left\|v\right\|=1}v^T\left(I-\eta_t\hat{H}_B^{-1}(x_t)\nabla^2f_B(x_t)\right)\left(I-\eta_t\hat{H}_B^{-1}(x_t)\nabla^2f_B(x_t)\right)^Tv \\
            \mathop{=}^{\tiny\mathcircled{1}} & \max_{\left\|v\right\|=1} v^T\left(I - \eta_t \hat{H}_B^{-1}(x_t)\nabla^2 f_B(x_t) - \eta_t\nabla^2f_B(x_t)\hat{H}_B^{-1}(x_t) +\eta_t^2\hat{H}_{B}^{-1}(x_t)\left(\nabla^2f_B(x_t)\right)^2\hat{H}_B^{-1}(x_t)\right)v \\
            \mathop{=}^{\tiny\mathcircled{2}} & \max_{\left\|v\right\|=1} v^T\left(I - 2\eta_t\hat{H}_B^{-1}(x_t)\nabla^2f_B(x_t)\right)v + \eta_t^2v^T\hat{H}_{B}^{-1}(x_t)\left(\nabla^2f_B(x_t)\right)^2\hat{H}_B^{-1}(x_t)v\\
            \le & \max_{\left\|v\right\|=1} \underbrace{v^T\left(I - 2\eta_t\hat{H}_B^{-1}(x_t)\nabla^2f_B(x_t)\right)v}_{part\ 1.1} + \max_{\left\|v\right\|=1} \underbrace{\eta_t^2v^T\hat{H}_{B}^{-1}(x_t)\left(\nabla^2f_B(x_t)\right)^2\hat{H}_B^{-1}(x_t)v}_{part\ 1.2}.
        \end{split}
    \end{equation}
    where $\tiny\mathcircled{1}$ and $\tiny\mathcircled{2}$ establish because of the symmetry of $\hat{H}_B^{-1}(x_t)$ and $\nabla^2 f_B(x_t)$. 
    
    \textbf{With $part\ 1.1$} in Eq.~\eqref{eq:part_1_relaxiation} and the construction of $\hat{H}_B(x_t)$, we have
    \begin{equation}
        \label{eq:part_1_equivalence}
        \begin{split}
            &I-2\eta_t\hat{H}_B^{-1}(x_t)\nabla^2f_B(x_t)\\
            =&I-2\eta_t\left(U\left(U^T\nabla^2f_B(x_t)U\right)^{-1}U^T + \lambda^{-1}U_{\perp}U_{\perp}^T\right)\left(UU^T+U_{\perp}U_{\perp}^T\right)\nabla^2f_B(x_t)\left(UU^T+U_{\perp}U_{\perp}^T\right)\\
            =&I-2\eta_t\left(UU^T + \underbrace{\lambda^{-1}\left(U_{\perp}U_{\perp}^T\right)\nabla^2f_B(x_t)UU^T}_{part\ 1.1.1} + \underbrace{U\left(U^T\nabla^2f_B(x_t)U\right)^{-1}U^T\nabla^2f_B(x_t)U_{\perp}U_{\perp}^T}_{part\ 1.1.2} + \right.\\
            & \left.\underbrace{\lambda^{-1}U_{\perp}U_{\perp}^T\nabla^2f_B(x_t)U_{\perp}U_{\perp}^T}_{part\ 1.1.3}\right),
        \end{split}
    \end{equation}
    where we denote 
    \begin{equation}
        \label{eq:proof_important_abbrevation}
        \epsilon_H = \left\|U_{\perp}U_{\perp}^T\nabla^2 f_B(x_t)\right\|,\quad v=Ua + U_{\perp}b\ and\quad U_{\perp}U_{\perp}^T = I- UU^T
    \end{equation}
    for abbreviation. Besides, for the definition of $v$, there are
    \begin{equation}
        \label{eq:v_normalization_relaxation}
        v^TU = a^TU^TU,\quad v^TU_{\perp} = b^TU_{\perp}^TU_{\perp}.
    \end{equation}
    
    \emph{For $part\ 1.1.1$}, we have
    \begin{equation}
        \label{eq:techniques_for_1_1}
        \begin{split}
            &-2\eta_t\lambda^{-1}v^T\left(U_{\perp}U_{\perp}^T\right)\nabla^2f_B(x_t)UU^Tv\\
            \mathop{\le}^{\tiny\mathcircled{1}}& \frac{2\eta_t}{\lambda}\left(\frac{\left\|v^TU_{\perp}U_{\perp}^T\nabla^2 f_B(x_t)\right\|^2}{2\rho\lambda} + \frac{\rho\lambda \left\|UU^Tv\right\|^2}{2}\right) \\
            \mathop{\le}^{\tiny\mathcircled{2}}& 2\eta_t\left(\frac{\left\|b^TU_{\perp}^T\right\|^2\cdot \left\|U_{\perp}U_{\perp}^T\nabla^2 f_B(x_t)\right\|^2}{2\rho\lambda^2} + \frac{\rho\left\|UU^T\right\|^2\cdot \left\|Ua\right\|^2}{2}\right)\\
            =&2\eta_t\left(\frac{\epsilon_H^2}{2\rho\lambda^2}\cdot \left\|b\right\|^2 + \frac{\rho}{2}\cdot\left\|a\right\|^2\right),
        \end{split}
    \end{equation}
    where $\tiny\mathcircled{1}$ and $\tiny\mathcircled{2}$ establish because of the \emph{Young's inequality} and the \emph{Cauchy–Schwarz inequality}. 
    
    \emph{For $part\ 1.1.2$}, with the similar proof technique, we obtain 
    \begin{equation}
        \label{eq:techniques_for_1_2}
        \begin{split}
            &-2\eta_tv^TU\left(U^T\nabla^2f_B(x_t)U\right)^{-1}U^T\nabla^2f(x_t)U_{\perp}U_{\perp}^Tv\\
            \le  & 2\eta_t\left(\frac{\rho\lambda_m^2\left\|v^TU\left(U^T\nabla^2f_B(x_t)U\right)^{-1}U^T\right\|^2}{2} + \frac{\left\|\nabla^2f_B(x_t)U_{\perp}U_{\perp}^Tv\right\|^2}{2\rho\lambda_m^2}\right) \\
            \le & 2\eta_t\left(\frac{\rho\lambda_m^2\left\|a^TU^T\right\|^2\cdot\left\|U\left(U^T\nabla^2 f_B(x_t)U\right)^{-1}U^T\right\|^2}{2} + \frac{\left\|\nabla^2 f_B(x_t)U_{\perp}U_{\perp}^T\right\|^2 \cdot \left\|U_{\perp}b\right\|^2}{2\rho\lambda_m^2}\right)\\
            = & 2\eta_t\left(\frac{\rho}{2}\cdot \left\|a\right\|^2 + \frac{\epsilon_H^{2}}{2\rho\lambda_m^2}\cdot \left\|b\right\|^2\right).
        \end{split}
    \end{equation}
    
    \emph{For $part\ 1.1.3$}, we can easily find 
    \begin{equation}
        \label{eq:techniques_for_1_3}
        \begin{split}
            -\lambda^{-1}v^TU_{\perp}U_{\perp}^T\nabla^2f_B(x_t)U_{\perp}U_{\perp}^Tv \le -\sigma_d/\lambda \cdot \left\|b\right\|^2.
        \end{split}
    \end{equation}
    
    Submitting Eq.\eqref{eq:techniques_for_1_1}, Eq.\eqref{eq:techniques_for_1_2} and Eq.\eqref{eq:techniques_for_1_3} back into Eq.\eqref{eq:part_1_equivalence}, we obtain that, for any feasible $v$, the $part\ 1.1$ of  Eq.~\eqref{eq:part_1_relaxiation} satisfies
    \begin{equation}
        \label{eq:constants_relation_1}
        \begin{split}
            & v^T\left(I - 2\eta_t\hat{H}_B^{-1}(x_t)\nabla^2f_B(x_t)\right)v\\
            \le & v^Tv - 2\eta_t\left(\left(1 - \rho\right)\cdot \left\|a\right\|^2 + \left(\frac{\sigma_d}{\lambda} - \frac{\epsilon_H^2}{2\rho\lambda_m^2} - \frac{\epsilon_H^2}{2\rho\lambda^2}\right)\cdot \left\|b\right\|^2 \right)\\
            \mathop{=}^{\tiny\mathcircled{1}} & \left(1 - 2\eta_t\left(1 - \rho\right)\right)\cdot \left\|a\right\|^2 + \left(1 - 2\eta_t\left(\frac{\sigma_d}{\lambda} - \frac{\epsilon_H^2}{2\rho\lambda_m^2} - \frac{\epsilon_H^2}{2\rho\lambda^2}\right)\right)\cdot \left\|b\right\|^2\\
            \le & \max\left\{1 - 2\eta_t\left(1 - \rho\right), 1 - 2\eta_t\left(\frac{\sigma_d}{\lambda} - \frac{\epsilon_H^2}{2\rho\lambda_m^2} - \frac{\epsilon_H^2}{2\rho\lambda^2} \right)\right\},
        \end{split}
    \end{equation}
    where $\tiny\mathcircled{1}$ establish because of the fact $\left\|v\right\|^2 = \left\|a\right\|^2 + \left\|b\right\|^2$. For any specific $\lambda_m$ and $\sigma_{m+1}$, there is existing some $\alpha \le 0.5$ which satisfies the inequality $\lambda_m \le \alpha^{-1}\sigma_{m+1}$. To clarify the proof procedure, we assume 
    \begin{equation}
        \label{eq:simplify_settings}
        \begin{split}
            \alpha = 0.5\quad \mathrm{and}\quad \lambda_m \le \alpha^{-1}\sigma_{m+1} = 2\sigma_{m+1}.
        \end{split}
    \end{equation}
    Notice that, the proof techniques presented in the following are not constrained by such specific $\alpha$. 
    To minimize our convergence constant, we required that
    \begin{equation}
        \label{coff_eq}
        \begin{split}
            \frac{\sigma_d}{\alpha \lambda_m} - \frac{\epsilon_H^2}{2\rho\lambda_m^2} - \frac{\epsilon_H^2}{2\rho\alpha^2\lambda_m^2} = 1-\rho,
        \end{split}
    \end{equation}
    where we set $\lambda = \alpha\lambda_m$. Hence, we request that the positive root about $\rho$ in Eq.\eqref{coff_eq} should satisfy $\rho_{+}\in \left(0, 1\right)$. Then, we present the analytic form of $\rho_{+}$ as
    \begin{equation*}
        \begin{split}
            \rho_{+} = \frac{1+\sqrt{\left(1-\frac{\sigma_d}{\alpha\lambda_m}\right)^2+\frac{2\left(\alpha^2+1\right)\epsilon_H^2}{\alpha^2\lambda_m^2} }-\frac{\sigma_d}{\alpha\lambda_m} }{2},
        \end{split}
    \end{equation*}
    and the sufficient condition for $\rho\in \left(0, 1\right)$ is 
    \begin{equation}
        \label{coff_sufficient_condition}
        \begin{split}
            & \left(1 - \frac{\sigma_d}{\alpha\lambda_m}\right)^2 + \frac{2\left(\alpha^2+1\right)^2\epsilon_{H}^2}{\alpha^2\lambda_m^2} \le \left(1+\frac{\sigma_d}{\left(\alpha+\delta_{+}\right)\lambda_m}\right)^2 \\
            \Longleftrightarrow \quad & 1 - \frac{2\sigma_d}{\alpha\lambda_m} +\frac{\sigma_d^2}{\alpha^2\lambda_m^2} + \frac{2\left(\alpha^2+1\right)\epsilon_H^2}{\alpha^2\lambda_m^2} \le 1 + \frac{2\sigma_d}{\left(\alpha+\delta_{+}\right)\lambda_m}+\frac{\sigma_d^2}{\left(\alpha+\delta_{+}\right)^2\lambda_m^2}\\
            \Longleftrightarrow \quad & \frac{1}{\alpha^2\lambda_m^2}\left(\sigma_d^2 + 2\left(\alpha^2+1\right)\epsilon_H^2 - \frac{\alpha^2\sigma_d^2}{\left(\alpha+\delta_+\right)^2}\right) \le \left(\frac{2}{\alpha+\delta_+}+\frac{2}{\alpha}\right)\frac{\sigma_d}{\lambda_m}\\
            \Longleftrightarrow \quad & \left(1-\frac{\alpha^2}{\left(\alpha+\delta_+\right)^2}\right)\sigma_d^2 + 2\left(\alpha^2+1\right)\epsilon_H^2 \le 2\left(\alpha+\frac{\alpha^2}{\alpha+\delta_+}\right)\sigma_d\lambda_m \\
            \Longleftarrow \quad & \left(1-\frac{\alpha^2}{\left(\alpha+\delta_+\right)^2}\right)\sigma_d^2 + 2\left(\alpha^2+1\right)\epsilon_H^2 \le 2k\left(\alpha + \frac{\alpha^2}{\alpha+\delta_+}\right)\sigma_d^2, \quad \lambda_m\ge k\sigma_d\\
            \Longleftrightarrow \quad & \epsilon_H^2 \le \left[\frac{\frac{20}{3}k\alpha - \frac{5}{9}}{2\left(\alpha^2+1\right)}\right]\cdot\sigma_d^2,\quad \lambda_m\ge k\sigma_d,\quad \delta_+ = 0.5\alpha.
        \end{split}
    \end{equation}
    According to Corollary~\ref{coro:ori_rlra_bound}, the sufficient condition for the establishment of the last equation of Eq.~\eqref{coff_sufficient_condition} can be formulated as
    \begin{equation}
        \label{eq:final_sufficient_condition}
        \begin{split}
            \lambda_m \ge k\sigma_d \ge k\sqrt{\frac{2\left(\alpha^2+1\right)}{\frac{20}{3}k\alpha-\frac{5}{9}}}\cdot (1.2\sigma_{m+1})
        \end{split}
    \end{equation}
    with probability at least $1-6e^{m-l}(l\ge m+4)$. Combining Eq.\eqref{eq:final_sufficient_condition} and Eq.\eqref{eq:simplify_settings}, when $k=1$, we have
    \begin{equation}
        \label{eq:var_relations_inequ}
        \begin{split}
            &\alpha^{-1} =2 > \sqrt{1.296} =  1.2k\sqrt{\frac{2\left(\alpha^2+1\right)}{\frac{20}{3}k\alpha-\frac{5}{9}}}\\
            \Longrightarrow\quad  &\alpha^{-1}\sigma_{m+1}\ge \lambda_{m} \ge k\sigma_d \ge k\sqrt{\frac{2\left(\alpha^2+1\right)}{\frac{20}{3}k\alpha-\frac{5}{9}}}\cdot (1.2\sigma_{m+1}).
        \end{split}
    \end{equation}
    To summarize these equations, we have the following conclusion. When $k=1$, $\alpha=0.5$ and $\lambda_m \le 2\sigma_{m+1}$, we can set $\lambda = 0.5\lambda_m$, and have $\rho_{+} \ge 0$,
    \begin{equation}
        \label{eq:rho_range}
        \begin{split}
            \rho_+ &\mathop{\le}^{\tiny\mathcircled{1}} \frac{1}{2}\left(1+\left(1+\frac{\sigma_d}{1.5\alpha \lambda_m}\right) - \frac{\sigma_d}{\alpha\lambda_m}\right) = 1-\frac{\sigma_d}{3\lambda_m}
        \end{split}
    \end{equation}
    with probability at least $1-6e^{m-l}(l\ge m+4)$. In Eq.\eqref{eq:rho_range}, $\tiny\mathcircled{1}$ is from Eq.~\eqref{coff_sufficient_condition}.
    $Part\ 1.1$ in Eq.\eqref{eq:part_1_relaxiation} has
    \begin{equation}
        \label{eq:constants_relation}
        \begin{split}
            \max_{\left\|v\right\|=1} v^T\left(I - 2\eta_t\hat{H}_B^{-1}(x_t)\nabla^2f_B(x_t)\right)v &= 1 - 2\eta_t\left(1 - \rho_+\right) \le 1-\frac{2\sigma_d}{3\lambda_m}\eta_t.
        \end{split}
    \end{equation}
    
    With Similar to the tricks on the $part\ 1.1$, \textbf{for $part\ 1.2$} in Eq.\eqref{eq:part_1_relaxiation}, we have
    \begin{equation}
        \label{eq:part_1_2_relaxiation}
        \begin{split}
            &\max_{\left\|v\right\|=1} \eta_t^2v^T\hat{H}_{B}^{-1}(x_t)\left(\nabla^2f_B(x_t)\right)^2\hat{H}_B^{-1}(x_t)v= \eta_t^2\left\|\hat{H}_B^{-1}(x_t)\nabla^2f_B(x_t)\right\|^2 \\
            =&\eta_t^2\left(\max_{\left\|v\right\|=1} v^T(\hat{H}_B^{-1}(x_t)\nabla^2f_B(x_t))v \right)^2\\
            \mathop{=}^{\tiny\mathcircled{1}}&\eta_t^2 \left[\max_{\left\|v\right\|=1} v^T\left(\underbrace{UU^T}_{part\ 1.2.1} + \underbrace{\lambda^{-1}\left(U_{\perp}U_{\perp}^T\right)\nabla^2f_B(x_t)UU^T}_{part\ 1.2.2} + \right.\right.\\
            & \left.\left.\underbrace{U\left(U^T\nabla^2f_B(x_t)U\right)^{-1}U^T\nabla^2f(x_t)U_{\perp}U_{\perp}^T}_{part\ 1.2.3} +\underbrace{\lambda^{-1}U_{\perp}U_{\perp}^T\nabla^2f_B(x_t)U_{\perp}U_{\perp}^T}_{part\ 1.2.4}\right)v \right]^2.
        \end{split}
    \end{equation}
    where $\tiny\mathcircled{1}$ establish because of Eq.~\eqref{eq:part_1_equivalence}. Utilizing the similar proof techniques presented in Eq.~\eqref{eq:techniques_for_1_1} and Eq.~\eqref{eq:techniques_for_1_2} on $part\ 1.2.2$ and  $part\ 1.2.3$, we respectively have
    \begin{equation}
        \label{eq:part_1_2_relaxiation_detailed_1}
        \begin{split}
            \lambda^{-1}v^T\left(U_{\perp}U_{\perp}^T\right)\nabla^2f_B(x_t)UU^Tv &\le \frac{\epsilon_H^2}{2\rho_+\lambda^2}\left\|b\right\|^2 + \frac{\rho_+}{2}\left\|a\right\|^2,\\
            v^TU\left(U^T\nabla^2f_B(x_t)U\right)^{-1}U^T\nabla^2f(x_t)U_{\perp}U_{\perp}^Tv &\le \frac{\rho_+}{2}\cdot \left\|a\right\|^2 + \frac{\epsilon_H^{2}}{2\rho_+\lambda_m^2}\cdot \left\|b\right\|^2,
        \end{split}
    \end{equation}
    where $a$, $b$ and $\epsilon_H$ are introduced as Eq.~\eqref{eq:proof_important_abbrevation}. Besides $\rho_+$ satisfies Eq.~\eqref{coff_eq}. With such settings, we have the following inequalities for $part\ 1.2.1$ and $part\ 1.2.4$
    \begin{equation}
        \label{eq:part_1_2_relaxiation_detailed_2}
        \begin{split}
            &v^T U U^T v = \left\|a\right\|^2\\
            &\lambda^{-1}v^TU_{\perp}U_{\perp}^T\nabla^2f_B(x_t)U_{\perp}U_{\perp}^T v \\
            \le& \lambda^{-1}\left[\frac{\epsilon_H}{2}\left\|v^T U_{\perp}U_{\perp}^T\right\|^2 + \frac{1}{2\epsilon_H}\left\|\nabla^2f_B(x_t)U_{\perp}U_{\perp}^T\cdot U_{\perp}U_{\perp}^T v\right\|^2\right]\\
            \le & \lambda^{-1} \left[\frac{\epsilon_H}{2}\left\|v^T U_{\perp}U_{\perp}^T\right\|^2 + \frac{1}{2\epsilon_H}\left\|\nabla^2f_B(x_t)U_{\perp}U_{\perp}^T\right\|^2\cdot \left\|U_{\perp}U_{\perp}^T v\right\|^2\right] = \frac{\epsilon_H}{\lambda} \left\|b\right\|^2.
        \end{split}
    \end{equation}
    Therefore, with Eq.~\eqref{eq:part_1_2_relaxiation}, Eq.~\eqref{eq:part_1_2_relaxiation_detailed_1} and Eq.~\eqref{eq:part_1_2_relaxiation_detailed_2}, we can obtain that
    \begin{equation}
        \begin{split}
            & \max_{\left\|v\right\|=1} \eta_t^2v^T\hat{H}_{B}^{-1}(x_t)\left(\nabla^2f_B(x_t)\right)^2\hat{H}_B^{-1}(x_t)v\\
            \le & \eta_t^2\left[\max\left\{1+\rho_+, \frac{\epsilon_H}{\lambda} + \frac{\epsilon_H^2}{2\rho_+\lambda_m^2} + \frac{\epsilon_H^2}{2\rho_+\lambda^2}\right\}\right]^2.
        \end{split}
    \end{equation}
    With the condition shown in Eq.~\eqref{eq:var_relations_inequ} and our requirements, we have
    \begin{equation}
        \label{coff_eq2}
        \frac{\epsilon_H}{\lambda_m}\le \left[k\sqrt{\frac{2\left(\alpha^2+1\right)}{\frac{20}{3}k\alpha-\frac{5}{9}}}\right]^{-1} = 1/\sqrt{1.296} \le 1\quad \mathrm{and}\quad \frac{\sigma_d}{\lambda_m}\le 1.
    \end{equation}
    Thus, we can derive
    \begin{equation}
        \begin{split}
            &\frac{\epsilon_H}{\lambda} + \frac{\epsilon_H^2}{2\rho_+\lambda_m^2} + \frac{\epsilon_H^2}{2\rho_+\lambda^2} \mathop{=}^{\tiny\mathcircled{1}} \rho_+ -1 + \frac{\sigma_d}{\lambda} + \frac{\epsilon_H}{\lambda} \mathop{\le}^{\tiny\mathcircled{2}} 4\left(1+\rho_+\right),
        \end{split}
    \end{equation}
     where $\tiny\mathcircled{1}$ and $\tiny\mathcircled{2}$ establish because of Eq.~\eqref{coff_eq} and Eq.~\eqref{coff_eq2} respectively. That is to say, we can scale $part\ 1.2$ in Eq.\eqref{eq:part_1_relaxiation} as
     \begin{equation}
         \label{eq:constants_relation_2}
         \max_{\left\|v\right\|=1} \eta_t^2v^T\hat{H}_{B}^{-1}(x_t)\left(\nabla^2f_B(x_t)\right)^2\hat{H}_B^{-1}(x_t)v \le 16\eta_t^2 \left(1 + \rho_+\right)^2.
     \end{equation}
     Therefore, considering Eq.~\eqref{eq:constants_relation_1}, Eq.~\eqref{eq:rho_range} and Eq.~\eqref{eq:constants_relation_2}, we have $part\ 1$ in Eq.~\eqref{eqn:thm42_eq2}
     \begin{equation}
        \label{eq:coef_part_1}
         \begin{split}
             \left\|I - \eta_t \hat{H}_B^{-1}(x_t)\nabla^2 f_{B_H}(x_t)\right\| \le \sqrt{1-\left(\frac{\sigma_d}{3\lambda_m}\right)^2\eta_t}\le 1-\frac{\sigma_d^2}{18\lambda_m^2}\eta_t,
         \end{split}
     \end{equation}
     when $\eta_t$ satisfies
     \begin{equation*}
         \begin{split}
             & 16\eta_t^2\left(1+ \rho_+\right)^2 \le \eta_t \left(1+\rho_+\right)\left(1-\rho_+\right), \eta_t\ge 0 \\
             \Longleftrightarrow & 0\le \eta_t\le \frac{1-\rho_+}{16\left(1+\rho_+\right)} \mathop{\Longleftarrow}^{\tiny\mathcircled{1}} 0\le \eta_t \le \frac{\sigma_d}{96\lambda_m - 16\sigma_d}.
         \end{split}
     \end{equation*}
     The sufficient condition $\tiny\mathcircled{1}$ comes from the upper bound of $\rho_+$ shown in Eq.~\eqref{eq:rho_range}.
    
    \textbf{For the $part\ 2$ in Eq.\eqref{eqn:thm42_eq2}:} We denote the stochastic bias matrix $\Delta_i (x_t)$ and the sub-sample hessian matrix $\nabla^2 f_{B}(x_t)$ as follows
    \begin{equation}
        \label{eqn:thm42_eq0}
        \begin{split}
            & \Delta_i (x_t) = \nabla^2 f_{i}(x_t) - \mathbb{E}\left[\nabla^2 f_{i}(x_t)\right]=\nabla^2 f_{i}(x_t) - \nabla^2 F(x_t),\\
            & \nabla^2 f_{B}(x_t) = \sum_{i\in B}\nabla^2 f_{i}(x_t).
        \end{split}
    \end{equation}
    With Assumption 2.3 and the construction Eq.\eqref{eqn:thm42_eq4}, we can get some properties
    \begin{equation*}
        \begin{split}
            & \mathbb{E}\left[\left\|\Delta_i(x)\right\|\right]=0,\qquad  \max\limits_{i\le n} \left\|\nabla^2 f_{i}(x_t)\right\|\le K,\\
            & \max\limits_{i\le n} \left\|\Delta_i(x_t)\right\| \le 2K,\qquad\max\limits_{i\le n} \left\|\Delta_i(x_t)^2\right\| \le 4K^2.  \\
        \end{split}
    \end{equation*}
    According to the \emph{matrix Bernstein's inequality} given in Lemma \ref{lem:matrix_concentration_inequality_1}, for any available $x_t$ we have bound the part 2 in~Eq.\eqref{eqn:thm42_eq2} as 
    \begin{equation}
        \label{eqn:thm42_eq3}
        \mathbb{P}\left[\left\|\nabla^2 f_{B_H}(x_t) - \nabla^2 F(x_t)\right\|\ge \tilde{\epsilon}\right]\le 2d\exp\left\{-\frac{\tilde{\epsilon}^2 b}{16K^2}\right\}.
    \end{equation}
    Hence, for any $\tilde{\epsilon}$, if we want $\|\nabla^2 f_{B}(x_t) - \nabla^2 F(x_t)\|\le \tilde{\epsilon}$ with probability at least $1-e^{m-l}$,  the sampled size of $B$, $b$, should satisfy
    \begin{equation}
        \label{eq:thm42_batch_size_selection}
        b \ge \min\left\{\frac{16K^2}{\tilde{\epsilon}^2}\cdot \left(l-m+\log\left(2d\right)\right), N\right\}.
    \end{equation}
    
    \textbf{For the $part\ 3$ in Eq.\eqref{eqn:thm42_eq2}:} With Assumption 2.1, the $part\ 3$ in~Eq.\eqref{eqn:thm42_eq2} can be bounded as follows
    \begin{equation}
        \label{eqn:thm42_eq4}
        \begin{split}
        & \left\|\nabla^2 F(x_t) - \int_0^1{\nabla^2 F(x^* + \tau (x_t - x^*))d\tau}\right\|  = \left\|\int_0^1{\nabla^2 F(x_t) - \nabla^2 F(x^* + \tau (x_t - x^*))d\tau}\right\|\\
        \le & \int_0^1 \left\|\nabla^2 F(x_t) - \nabla^2 F(x^* + \tau (x_t - x^*))\right\| d\tau\\
        \le & \int_0^1 M_{b}(1-\tau)\left\|x_t - x^*\right\|d\tau = \frac{M_{b}}{2}\left\|x_t - x^*\right\|.
        \end{split}
    \end{equation}
    Hence, the Eq.\eqref{eqn:thm42_eq2} can be rewrited in the following with probability at least $1-6e^{m-l}(l\ge m+4)$ as
    \begin{equation}
        \begin{split}
            & \left\|I - \eta_t \hat{H}_B^{-1}(x_t)\int_0^1{\nabla^2 F(x^*+ \tau(x_t - x^*))(x_t - x^*)d\tau}\right\|\\
            \le & 1-\frac{\sigma_d^2}{18\lambda_m^2}\eta_t + \eta_t\left\|\hat{H}_B^{-1}(x_t)\right\| \left(\tilde{\epsilon} +\frac{M_{b}}{2}\left\|x_t - x^*\right\|\right)
        \end{split}
    \end{equation}
    due to Eq.~\eqref{eq:coef_part_1}, Eq.~\eqref{eqn:thm42_eq3}, Eq.~\eqref{eq:thm42_batch_size_selection} and Eq.~\eqref{eqn:thm42_eq4}. Since $\lambda\le \min\{\sigma_{m+1}^t, \sigma_{\min}(U^T\nabla f_{B}(x_t)U)\}$, we can conclude the convergence rate as
    \begin{equation}
        \begin{split}
            \left\|x_{t+1} - x^*\right\| & \le  c_1^t\left\|x_t - x^*\right\| + c_2^t \left\|x_t - x^*\right\|^2 \\
            c_1^t &=1-\frac{\sigma_d^2}{18\lambda_m^2}\eta_t + \frac{\tilde{\epsilon}\eta_t}{\lambda},\quad c_2^t =  \frac{M_{b}\eta_t}{2\lambda}
        \end{split}
        \label{eqn:a5}
    \end{equation}
    with probability at least $1-6e^{m-l}(l\ge m+4)$. For clarify the coefficients of convergence further, we have
    \begin{equation}
        \begin{split}
            \frac{\tilde{\epsilon}\eta_t}{\lambda} \le \frac{\sigma_d^2}{36\lambda_m^2} \Longleftarrow \quad b = \Theta\left(K^2\lambda_m^2\sigma_d^{-4}\log(d)\right) = \Theta\left(K^2\sigma_{m+1}^2\sigma_d^{-4}\log(d)\right).
        \end{split}
    \end{equation}
    Thus, the coefficients of convergence satisfy
    \begin{equation}
        \begin{split}
            c_1^t &=1-\frac{\sigma_d^2}{36\lambda_m^2}\eta_t,\quad c_2^t =  \frac{M_{b}\eta_t}{2\lambda}
        \end{split}
        \label{eqn:a5}
    \end{equation}
\end{proof}

\section{Additional Experiments}
\subsection{Exact Experimental setting on Logistic Regression}
\label{sec:ex_lr}
In this section, we detail our experiments on Logistic Regression. We set the objective function as
\begin{equation*}
        \min\limits_{x} -\frac{1}{N}\mathop{\sum}\limits_{i=1}^N\log\frac{1}{1 + \exp\left(-y_i\theta_i^Tx\right)} + \frac{a}{2} \left\|x\right\|^2.
\end{equation*}

\textbf{Datasets}: In this experiment, we use two datasets for the binary classification task including \emph{MNIST}~\cite{mnist19} and \emph{CovType}~\cite{covertype19}. 
\begin{itemize}
    \item \emph{MNIST}: The MNIST database is a dataset of handwritten digits. It has 60,000 training samples, and 10,000 test samples. Each image is represented by 28x28 pixels, each containing a value 0 - 255 with its grayscale value.
    \item \emph{CovType}: Predicting forest cover type from cartographic variables only (no remotely sensed data). The actual forest cover type for a given observation (30 x 30 meter cell) was determined from US Forest Service (USFS) Region 2 Resource Information System (RIS) data. Independent variables were derived from data originally obtained from US Geological Survey (USGS) and USFS data. Data is in raw form (not scaled) and contains binary (0 or 1) columns of data for qualitative independent variables (wilderness areas and soil types). 
\end{itemize}  

\textbf{Data Preprocessing}: For the two datasets proposed above, we will list our data preprocessing in the following
\begin{itemize}
    \item \emph{MNIST}: We first select the samples whose label is $4$ or $9$ (Logistic Regression is used in binary classification tasks). Then we normalize the samples as
    \begin{equation*}
        \theta_{ij} = \theta_{ij}\Bigg/\sqrt{\left(\mathop{\sum}\limits_{k=1}^d \theta^2_{ik}\right)}.
    \end{equation*}
    where $\theta_{ij}$ means the feature $j$ of the sample $i$.
    \item \emph{CovType}: We use the total samples, and normalized the samples as we did for MNIST.
\end{itemize}

\textbf{Preiteration:} As the preiteration requirement of LiSSA~\cite{agarwal2017second}, for the sake of fairness, we will perform 2 epoch SVRG~\cite{johnson2013accelerating} iteration for all the methods after initializing the decision variable $x_0$.

\textbf{Experimental Results}: We have shown our experimental results about Logistic Regression in our paper.

\textbf{Parameters Selection}: We have shown all the experimental parameters in Path: ./code/lr$\_$data/params

\subsection{Extra Experiments on Huber SVM}
\label{sec:ex_exp}
In this section, we conduct our experiments on Support Vector Machines. To satisfy assumptions of all baselines and \method, we set the loss function $l(.)$ in the SVM
\begin{equation*}
    \min\limits_{x}\ \frac{1}{2}a\left\|x\right\|^2 +\frac{1}{N}\mathop{\sum}\limits_{i=1}^{N}l(y_i, \theta_i^Tx)
\end{equation*}
as smoothed Huber loss presented as
$$
    l(y, \theta^Tx) = \begin{cases}
            0, &\text{if }\ y\cdot \theta^Tx>3/2,\\
            \frac{1}{2}(\frac{3}{2}- y\cdot \theta^Tx)^2, &\text{if }\ \left|1-y\cdot \theta^Tx\right|<1/2, \\
            1-y\cdot \theta^Tx, &\text{otherwise}.
    \end{cases}
$$
\textbf{Datasets}: In this experiment, we use two text classification datasets including \emph{webspam}~\cite{webb2006introducing} and \emph{Yelp Review Polarity}~\cite{yelpdataset19}. \begin{itemize}
    \item \emph{webspam}: For each instance, continuous 1 bytes are  treated as a word, and use word count as the feature value.
    In the end, there are 254 features for each sample, and the number of sample is about $160,000$.
    \item \emph{Yelp Review Polarity}: The Yelp reviews polarity dataset is constructed by considering stars 1 and 2 negative, and 3 and 4 positive. 
    For each polarity 280,000 training samples and 19,000 testing samples are take randomly. 
    In total there are 560,000 trainig samples and 38,000 testing samples.
\end{itemize}  

\textbf{Data Preprocessing}: For the two datasets proposed above, we will list our data preprocessing in the following
\begin{itemize}
    \item \emph{webspam}: We normalize the samples as
    \begin{equation*}
        \theta_{ij} = \theta_{ij}\Bigg/\sqrt{\left(\mathop{\sum}\limits_{k=1}^d \theta^2_{ik}\right)}.
    \end{equation*}
    where $\theta_{ij}$ means the feature $j$ of the sample $i$.
    \item \emph{Yelp Review Polarity}: We want all baselines (including the second-order methods) to converge in a relatively short period of time. Hence, we use \emph{CountVectorizer} to preprocess the data, and normalize each sample as we do in \emph{webspam}. Then, we uniformly pick $1000$ features to describe each sample. If the number of samples which are described as the zero vector is over a half of the total sample number, we will re-sample the features 
\end{itemize}

\textbf{Preiteration:} As the preiteration requirement of LiSSA~\cite{agarwal2017second}, for the sake of fairness, we will perform 2 epoch SVRG~\cite{johnson2013accelerating} iteration for all the methods after initializing the decision variable $x_0$.

\textbf{Experimental Results}: We first show the graph about the training loss versus running time in different condition numbers in Fig~\ref{fig:training_loss_fig}. Second, we give the Hessian approximation error to validate the quality of our approximate Hessian in Fig~\ref{fig:hessian_approximation_fig}.
\begin{figure*}[tb]
    \centering
    \includegraphics[width=1.1\textwidth]{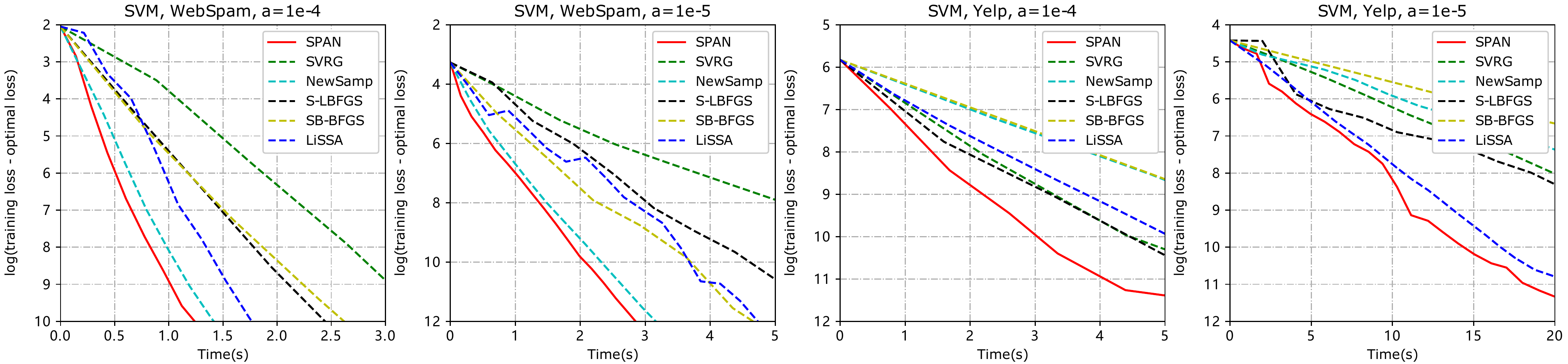}
    \vspace*{-20pt}
    \caption{Training loss versus running time. The first two columns are for webspam dataset. The last two are for Yelp Review Polarity dataset. Note \method achieves the best or near the best performance with respect to wall-clock time.}
    \label{fig:training_loss_fig}
\end{figure*}
\begin{figure*}[tb]
    \centering
    \includegraphics[width=1.1\textwidth]{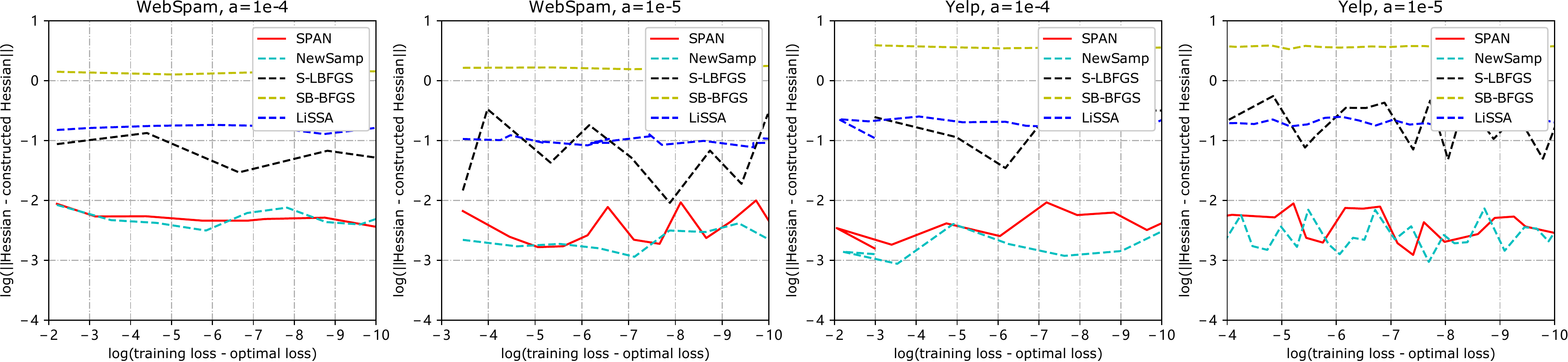}
    \vspace*{-20pt}
    \caption{Hessian approximation error. Note that \method achieves near lowest error.}
    \label{fig:hessian_approximation_fig}
\end{figure*}

\textbf{Parameters Selection}: We have shown all the experimental parameters in Path: ./code/svm$\_$data/params

\end{document}